\documentclass[reqno,11pt,a4]{amsart}
\usepackage{amsthm,amsfonts,amssymb,amsmath,latexsym,marginnote,mathtools}
\usepackage[dvips]{graphicx}
\usepackage[colorlinks=true, pdfstartview=FitV, linkcolor=BrickRed,citecolor=black, urlcolor=black]{hyperref}
\usepackage{enumitem}
\usepackage[usenames,dvipsnames]{color}
\usepackage{comment}
\usepackage[margin=2.5cm]{geometry}

\numberwithin{equation}{section}

\newtheorem{theorem}{Theorem}[section]
\newtheorem{lemma}[theorem]{Lemma}
\newtheorem{proposition}[theorem]{Proposition}

\theoremstyle{definition}
\newtheorem{remark}[theorem]{Remark}

\def\R {\mathbb{R}}
\def\Z {\mathbb{Z}}
\def\N {\mathbb{N}}
\def\T {\mathbb{T}}

\newcommand{\eps}{\varepsilon}

\providecommand{\ip}[1]{\langle#1\rangle}
\providecommand{\abs}[1]{\left\lvert#1\right\rvert}

\providecommand{\norm}[1]{\left\|#1\right\|}

\title[On the stability of vacuum in the screened Vlasov-Poisson equation]{On the stability of vacuum \\
in the screened Vlasov-Poisson equation
}

\author{Mikaela Iacobelli}
\address{ETH Zurich and IAS Princeton}
\email{mikaela.iacobelli@math.ethz.ch, miacobelli@ias.edu}

\author{Stefano Rossi}
\address{ETH Zurich}
\email{stefano.rossi@math.ethz.ch}

\author{Klaus Widmayer}
\address{University of Vienna and University of Zurich}
\email{klaus.widmayer@univie.ac.at, klaus.widmayer@math.uzh.ch}

\begin{document}
\begin{abstract}
 We study the asymptotic behavior of small data solutions to the screened Vlasov-Poisson equation on $\R^d\times\R^d$ near vacuum. We show that for dimensions $d\geq 2$, under mild assumptions on localization (in terms of spatial moments) and regularity (in terms of at most three Sobolev derivatives) solutions scatter freely. In dimension $d=1$, we obtain a long time existence result in analytic regularity.
\end{abstract}
\maketitle

\tableofcontents
\section{Introduction}

This article is devoted to the study of the stability of vacuum in the screened Vlasov-Poisson system. 
This system models the behavior of interacting charged particles in plasmas, where long-range Coulomb forces are truncated by {Debye shielding}, see for instance \cite[Chapter 3]{Davidson1972} and \cite[Chapter 2]{Chen2016}. In addition, as discussed for example in \cite[Chapter 5]{BinneyTremaine2008} or \cite[Section 2.2]{Vogelsberger2014}, it is also used in astrophysical contexts to describe self-gravitating systems with {gravitational screening}, such as galaxies within dark matter halos. 

Here, we study the evolution of initially 
small particle distribution functions\footnote{The corresponding physically relevant, non-negative particle distribution is $f(x,v,t)=\mu^2(x,v,t)\geq 0$. For our arguments the $L^2$ based setting for $\mu$ is more technically convenient -- see \cite{IPWW2020} and also \cite{FOPW2024}.} $\mu:\R^d\times\R^d\times\R\to\R$, $d\geq 1$, satisfying
\begin{equation}\label{eq:scrVP}
\begin{aligned}
    \partial_t \mu +v\cdot\nabla_x\mu-\nabla_x\phi\cdot\nabla_v \mu&=0,\\
    (1-\Delta)\phi(x,t)=\rho(x,t)&=\int_{\R^d}\mu^2(x,v,t)dv.
\end{aligned}    
\end{equation}
Compared to the classical Vlasov-Poisson system, the standard Coulomb interaction kernel is replaced by screened interactions. This is particularly relevant from the point of view of ions in a massless electron approximation \cite{BGNS2018}, for example. In such contexts, a more complicated, nonlinear version of the Poisson equation in the second line of \eqref{eq:scrVP} is sometimes also considered (see e.g.\ \cite{GPI2021, B91, HK11, HKI17, CI23, GPI21, GI23}), for which not vacuum but spatially homogeneous (and suitably normalized) solutions can be natural stationary configurations.
Our focus, however, is on the essential features of screening in the absence of homogeneous backgrounds, which are captured adequately by \eqref{eq:scrVP}. 

In contrast to the classical Vlasov-Poisson equations (for electrons, where $-\Delta\phi=\rho$), screening de-singularizes the Poisson equation for the potential $\phi$, and is thus generally expected to lead to ``better-behaved'' dynamics. Nevertheless, even in the setting of screened interactions the long-time behavior of solutions to \eqref{eq:scrVP} is largely unknown. In the special case of dynamics near vacuum, in a first approximation it is natural to consider linearized dynamics, i.e.\ the behavior of solutions to the free transport equation, and then to ask whether nonlinear effects are of perturbative nature.\footnote{Hereby the sign in front of the nonlinear term in the first equation in \eqref{eq:scrVP} is of no consequence. For simplicity, we have chosen to stick to the plasma physical convention.} The answer to this latter question depends crucially on the time decay rates of the force field $\nabla_x\phi$ in \eqref{eq:scrVP}, which in turn depend on the dimension and on the form of the Poisson law in \eqref{eq:scrVP}. 

In all this, screened interactions are mathematically of particular interest 
in low dimensions  $d\leq 2$ \cite[Section 6]{Bed2022}, where the relevant stability mechanisms are comparatively weak. Indeed, in such settings, the corresponding classical Vlasov-Poisson equations are ill-understood, even near vacuum.\footnote{In $d=2$, to the best of our knowledge the strongest result to date concerns long-time (but not asymptotic) dynamics for a class of symmetry restricted initial data \cite{BAMP2024}, while in $d=1$ we are not aware of results going beyond local well-posedness. Moreover, recent works for the classical Vlasov-Poisson equations in $d=3$ suggest that certain spatially homogeneous configurations may be (quantifiably) ``more'' stable than vacuum -- see e.g.\ \cite{IPWW2024}.} As a consequence, whether or not nonlinear effects are of perturbative nature becomes a question of a delicate balance between linear decay mechanisms and nonlinear propagation of the ingredients (such as localization and regularity) required for their control. In addition to its main results for \eqref{eq:scrVP}, we hope that this article will contribute to an improved overall understanding of the stability mechanisms near vacuum in the unconfined setting, and in particular on the interplay of dimensions, localization and regularity.

Our first main result shows that the asymptotic behavior of this system in dimensions $d\geq 2$ is a free scattering dynamic. 
\begin{theorem}
\label{intro:thm1}
 Let $d\geq 2$. There exists $\eps_0>0$ such that if $\mu_0:\R^d\times\R^d\to\R$ 
 satisfies
 \begin{alignat}{3}
      &\textnormal{if }d\geq 3: && \quad \|\ip{x}^m \mu_0\|_{L^\infty_{x,v}(\R^d \times \R^d)}
      + \|\nabla_{x,v}\mu_0\|_{L^\infty_{x,v}(\R^d \times \R^d)}
      \le \eps\leq \varepsilon_0,&& \quad m>d,\label{eq:assump_3d}\\
      &\textnormal{if }d=2:&& \quad \|x \mu_0\|_{L^2_{x,v}(\R^2\times \R^2)}
        + \|\mu_0\|_{H^3_{x,v}(\R^2 \times \R^2)} \le \eps\leq\varepsilon_0,&&\label{eq:assump_2d}
 \end{alignat}
 then there exists a unique, global in time solution $\mu \in C_t(\R,W^{1,\infty}_{x,v}(\R^d \times \R^d))$, $d\geq 3$, resp.\ $\mu \in C_t(\R,H^3_{x,v}(\R^2 \times \R^2))$ to \eqref{eq:scrVP} for with $\mu(x,v,0)=\mu_0(x,v)$. 
 Moreover, the associated force fields decay as 
 \begin{equation}\label{eq:overall_decay}
   \norm{\nabla_x\phi(t)}_{L^\infty(\R^d)}\lesssim \eps^2\ip{t}^{-d-1}, 
     d\geq 3, \qquad 
   \norm{\nabla_x\phi(t)}_{L^\infty(\R^2)}\lesssim \eps^2\ip{t}^{-3+}, 
 \end{equation}
 and $\mu$ converges along trajectories of the linearized flow: there exists $\gamma_\infty \in L^\infty_{x,v}(\R^d \times \R^d)$, $d\geq 3$, resp.\ $\gamma_\infty \in H^1_{x,v}(\R^2 \times \R^2)$ such that
 \begin{equation}\label{eq:main_freescat}
     \mu(x+tv,v,t)\to \gamma_\infty(x,v),\quad t\to\infty, \quad \text{in }L^\infty_{x,v} \text{ resp.\ in } H^1_{x,v}.
 \end{equation}
\end{theorem}

\begin{remark}
We highlight some points of interest:
 \begin{enumerate}
  \item Decay estimates for small data solutions to the screened Vlasov-Poisson equation in $d \ge 2$ go back to \cite{SL2007,CS2011,D22}, where sharp decay was obtained in \cite{D22} using the vector fields method  of \cite{S16}. In a setting of comparatively high regularity, this implies global stability of vacuum via decay of the density and its derivatives (compare \eqref{in1:d3}-\eqref{in2:d3}, \eqref{in:lin3}-\eqref{in:lin2} and Remarks \ref{rem:sharpened3d}, \ref{rem:sharpened2d} with \cite[Remark 1.2]{D22}). In contrast, the present work builds on the approach developed in \cite{IPWW2020,FOPW2024}, which allows for weaker assumptions on the initial data and -- in our view -- clarifies the relevance of the various moment and derivative controls.

 \item The scattering statement \eqref{eq:main_freescat} corresponds to asymptotics as in the linearized problem and is thus also known as free scattering. The key stability mechanism in our ``unconfined'' setting with $x\in\R^d$ is the dispersive decay of the spatial density $\rho$ and associated field $\phi$ as in the free transport equation. This occurs at a dimension-dependent rate and requires localization of $\mu$. As is clear from the point of view of characteristics, free scattering is implied if the force field $\norm{\nabla_x\phi(t)}_{L^\infty}$ in \eqref{eq:scrVP} decays at a rate faster than $t^{-2}$, which we show in \eqref{eq:overall_decay}. 
 \item A more precise version of this theorem is given later in Theorem \ref{sec2:main"} for $d\geq 3$ and Theorem \ref{sec3:main} for $d=2$. We highlight that no velocity moments are necessary, and regularity can be propagated. In particular, under stronger assumptions on the initial data the scattering statement holds in a stronger topology and the decay estimate \eqref{eq:overall_decay} can be improved to the sharp rate $\norm{\nabla_x\phi(t)}_{L^\infty(\R^d)}\lesssim \eps^2\ip{t}^{-d-1}$, $d\geq 2$ -- see also Remarks \ref{rem:sharpened3d} and \ref{rem:sharpened2d}. Moreover, each derivative of $\phi$ decays at a one order faster rate, provided sufficient regularity and localization are controlled.
 
 \item In the un-screened case, where the potential satisfies
 $-\Delta\phi=\rho$ rather than the second equation in \eqref{eq:scrVP}, the asymptotic behavior near vacuum (under similar assumptions as in Theorem \ref{intro:thm1}) is only understood in dimensions $d\geq 3$. For $d\geq 4$ this is given by free scattering \cite{P23}, whereas $d=3$ is scattering-critical: along the characteristics a logarithmic in time correction with an asymptotic electric field profile is necessary in order to understand asymptotic convergence 
 \cite{CK16,IPWW2020,FOPW2024,BVR24}. This is due to the weaker dispersive decay rates of $\Delta_x^{-1}\rho$ as compared to $(1-\Delta_x)^{-1}\rho$ and their derivatives -- see also the discussion in Section \ref{sec:intro_lin} below.
 \item It is also natural to consider \eqref{eq:scrVP} in front of a fixed background, whereby the relevant density $\rho$ is replaced by $\rho-n_0$ for some constant background density $n_0>0$ (and as already mentioned additional nonlinear terms may be considered to describe ion dynamics more precisely). This allows for a large family of spatially homogeneous steady states $\mu(x,v,t)=M_0(v)$. An important class within these are the Penrose stable equilibria \cite{Penrose60}, for which asymptotic stability results are available in dimension $d\geq 2$ \cite{BMM18,HKNR21,HHX22, HNX24}. Similarly, these rely on dispersive mechanisms and establish free scattering of solutions under suitable assumptions on the initial data.
 
 \item In the confined setting $(x,v)\in\T^d\times\R^d$, the free transport dynamics give rise to phase mixing, whereby regularity of $\mu$ can be traded for decay of $\rho, \phi$ etc, independently of the dimension $d$. The stability of vacuum (with asymptotic free scattering dynamic) is then a consequence of the stability of Penrose stable equilibria. In contrast to the present finite regularity setting, though, this requires analytic-type (Gevrey-$3$) regularity \cite{MV11,BMM2016, GNR2021,IPWW2025}, see also \cite{GI23} for the ionic case.

 \item Another interesting perspective concerns the setting of the Vlasov-Riesz system, where the screened Coulomb kernel of \eqref{eq:scrVP} is replaced by a Riesz kernel. While this does not affect the decay rate of the density $\rho$, it leads to a different decay behavior for $\phi$. The recent work \cite{HK24} studies this question for a family of such kernels on $\R^3\times\R^3$, and even establishes asymptotic dynamics when a polynomial in time correction to the characteristics is required. 

 \item 
 Concerning the final data problem where the asymptotic state is prescribed,
 in the un-screened Vlasov-Poisson system with $x \in \R^d$, a modified scattering final data result is proved in \cite{FOPW2024} for $d=3$. 
 On the contrary, in the confined case of the torus $x \in \mathbb{T}^d$ a free scattering result holds in any dimension requiring analytic-type regularities of the asymptotic state. While the results in \cite{IPWW2025, BCGIR24} assume final data of small size, the only results available with final data of size $\mathcal{O}(1)$ are \cite{CM98, G24, BCR22}.
 \end{enumerate}
\end{remark}

In the one-dimensional setting $d=1$, sharp decay of $\partial_x\phi$ hinges on control of the localization of a velocity derivative of $\mu$, which seems to lead to a two-fold derivative loss, which stands in the way of our understanding of asymptotic dynamics. Instead, we obtain the long-time stability of initially small and analytic solutions:
\begin{theorem}[Long-time stability in $d=1$]\label{intro:thm1d}
    There exists $\eps_0>0$ such that the following holds: Given $R>0$ and $\mu_0$ satisfying
    \begin{equation*}
        \sum_{a\in\{0,1\}}\sum_{n=0}^\infty \frac{R^n}{n!} \left(\norm{\partial_x^{n+a}\mu_0}_{L^2_{x,v}}
        + \norm{\partial_v^{n+a}\mu_0}_{L^2_{x,v}}\right) \leq \eps \le \varepsilon_0,
    \end{equation*}
    there exist $T\gtrsim R^{2}\eps^{-4}$ and a unique solution $\mu\in C_t([0,T],\cap_{k=0}^\infty H^k_{x,v}(\R\times\R))$ to the screened Vlasov-Poisson equation \eqref{eq:scrVP} with $\mu(x,v,0)=\mu_0(x,v)$. Moreover, $\gamma(x,v,t):=\mu(x+tv,v,t)$ satisfies
    \begin{equation*}
    \sum_{a\in\{0,1\}}\mathcal{E}_v[\partial_v^a\gamma(t)]+ \mathcal{E}_x[\partial_x^a\gamma(t)]\lesssim \eps,\qquad 0\leq t\leq T,
    \end{equation*}
    where
 \begin{equation}\label{eq:intro_ana_norms}
     \mathcal{E}_{v}[f]:= \sum_{n=0}^\infty \frac{\lambda^n_t}{n!}\norm{\partial_v^n f}_{L^2_{x,v}}, \qquad
     \mathcal{E}_{x}[f]:= \sum_{n=0}^\infty \frac{\lambda^n_t}{n!}\norm{\partial_x^n f}_{L^2_{x,v}}
 \end{equation}
 with $\lambda_t=\lambda_t(R,\eps) := R-C\eps^2\ip{t}^{\frac12}$ and $C>0$ a universal constant.
\end{theorem}
\begin{remark}
 We comment on some points of relevance:
 \begin{enumerate}
     \item We work with analytic norms that incorporate a radius of analyticity $\lambda_t$ that decreases over time, since this allows for a loss of one derivative in energy type estimates, enabling us to improve over the trivial rate $t^{-1}$ of decay for the force field $\partial_x\phi$ to obtain the faster rate $t^{-3/2}$. However, with only energy type norms at disposal it does not seem feasible to obtain the sharp decay, and can thus only close an estimate on a finite time scale. This time scale is determined by the fact that the terms with the slowest time decay are also the ones with highest derivative loss, and are thus only controlled through an appropriately shrinking radius of analyticity, as is classical in such arguments \cite{N72, N77}.
     \item After the publication online of this work, a preprint \cite{Wei24} of D.\ Wei appeared, demonstrating a conjecture that we had formulated in the prior version of this remark: small, Gevrey-$2$ regular initial data lead to global solutions with sharp decay $\norm{\partial_x\phi}_{L^\infty}\lesssim \eps^2\ip{t}^{-2}$ and give rise to a modified scattering dynamic as in the un-screened Vlasov-Poisson equations in $d=3$.
 \end{enumerate}
\end{remark}

\subsection{Discussion of proofs}\label{sec:intro_proofs}
To understand the long-time dynamics it is convenient to filter out the free transport dynamic $\partial_t+v\cdot\nabla_x$ in \eqref{eq:scrVP} and work with the unknown
\begin{equation*}\label{def:gamma}
    \gamma(x,v,t):=\mu(x+tv,v,t),
\end{equation*}
which satisfies
\begin{equation}
\label{eq:gamma}
\begin{aligned}
    \partial_t \gamma(x,v,t)&=\nabla_x \phi(x+tv,t) \cdot \{\nabla_v - t\nabla_x\} \gamma(x,v,t),\\
    (1-\Delta)\phi(x,t)&=\rho(x,t)=\int_{\R^d}\gamma^2(x-vt,v,t) dv,
\end{aligned}
\end{equation}
and has the same initial data $\gamma(x,v,0)=\mu_0(x,v)$ as \eqref{eq:scrVP}. 
We will work with this equation to establish asymptotic resp.\ long-time dynamics of  $\gamma$ in $d\geq 2$ resp.\ $d=1$. The essential mechanism hereby is the decay over time of the spatial density $\rho$ and derived quantities, in particular the force field $\nabla_x\phi$. In the spatially unbounded setting ($x\in\R^d$) such decay occurs due to the dispersive effects of the free transport equation, at a rate that depends on the dimension $d$ and relies on localization. As is clear from \eqref{eq:gamma}, the rate of decay plays a decisive role in determining the time span over which solutions can be controlled. The main challenge is thus to understand the interplay of the following two key ingredients:
\begin{enumerate}[wide]
    \item \emph{Linear decay (see also Section \ref{sec:intro_lin}).} Given a function $h:\R^d\times\R^d\to\R$,
    consider the associated density and field
    \begin{equation*}
        \rho_h(x,t):=\int h^2(x-tv,v)dv,\qquad \phi_h:=(1-\Delta)^{-1}\rho_h,
    \end{equation*}
    and quantify requirements on $h$ to obtain (optimal) decay rates of $\rho_h$, $\nabla_x\phi_h$ etc. As is well known, the underlying mechanisms for decay rely on localization and regularity, which need to be captured here in a way that is minimally demanding for the nonlinear analysis.
    \item \emph{Nonlinear propagation of localization and regularity (see also Section \ref{sec:intro_nonlin}).} Consider an initially small, suitably regular and localized solution $\gamma$ to \eqref{eq:gamma}, and use the linear decay to propagate localization and regularity of $\gamma$. As is apparent from \eqref{eq:gamma}, one needs to balance these arguments well, since e.g.\ velocity regularity of $\gamma$ comes at a cost of time decay for derivatives of $\nabla_x\phi$, which in turn require additional velocity regularity of $\gamma$. On which time span can the corresponding bootstrap be closed? 

    The simplest setting hereby is that of dimensions $d\geq 3$. A simple bootstrap involving uniform (in time) bounds on $x$-moments as well as one order of derivatives on $\gamma$ and the comparatively rapid decay of $\nabla_x\phi$ then shows that small initial data in \eqref{eq:gamma} lead to solutions that scatter freely (see Section \ref{sec2}). In contrast, for $d=2$ the nonlinear arguments are much more delicate: we obtain a hierarchy of slowly growing bounds for velocity regularity (and $x$-moments) via multilinear, $L^2$-based estimates as introduced in \cite{IPWW2020}, combined with uniform control of a lower regularity $Z$-norm. The associated nonlinear bootstrap just barely closes in finite regularity $H^3_{x,v}$ -- see Section \ref{sec3}. Finally, in $d=1$ we are confronted with a dual loss of derivatives: one required to obtain linear decay and one in the equation for $\gamma$. We are then only able to close a nonlinear bootstrap involving analytic regularity on a finite, but long time interval -- see Section \ref{sec:1d}. In particular, we do not obtain sharp decay rates for $\nabla_x\phi$.
\end{enumerate}
\medskip
We discuss next these two points in some more detail.

\subsubsection{Linear decay}\label{sec:intro_lin}
A simple change of variables $v\mapsto x-tv=:a$ shows that the density $\rho_h(x,t)$ decays at the rate $\mathcal{O}(t^{-d})$, i.e.\
\begin{equation}\label{eq:rho_intro}
 \rho_h(x,t)=t^{-d}\int h^2\Bigl(a,\frac{x-a}{t}\Bigr)da, 
\end{equation}
provided that $h$ is suitably localized. Importantly, this formula also shows that higher derivatives $\partial_x^\alpha\rho$ of $\rho$ decay at faster rates $\mathcal{O}(t^{-d-\abs{\alpha}})$, provided that $h$ is sufficiently regular in $v$, e.g.\
\begin{equation*}
 \nabla_x\rho_h(x,t)=2t^{-d-1}\int h\Bigl(a,\frac{x-a}{t}\Bigr)\nabla_vh\Bigl(a,\frac{x-a}{t}\Bigr)da.   
\end{equation*}
In our screened setting, these arguments transfer to $\phi_h$, since the map $\rho_h\mapsto\phi_h=(1-\Delta)^{-1}\rho_h$ is bounded on $L^p(\R^d)$, $1\leq p\leq \infty$ (see e.g.\ \cite[Chapter 5, \S 3]{S1970}). Heuristically, one thus expects that the sharp decay of the force field is $\norm{\nabla_x\phi_h}_{L^\infty}= \mathcal{O}(t^{-d-1})$, which can allow for the right hand side of \eqref{eq:gamma} to be time integrable whenever $d\geq 2$, and hints at the free scattering dynamic. However, this requires a regime in which one can nonlinearly propagate sufficient regularity and smoothness. It is thus important to establish a larger class of decay estimates that also give slower decay rates under less restrictive assumptions.

To make these arguments precise, we observe that
\begin{equation*}
 \phi_h(x,t)=(1-\Delta)^{-1}\rho_h(x,t)=G_d\ast \rho_h(x,t),
\end{equation*}
where $G_d$ is the Green function of $(1-\Delta)$ on $\R^d$, given by\footnote{This is a particular example of Bessel kernel, see e.g.\ \cite[Chapter 5]{S1970}.}
\begin{equation}
\label{green}
    G_d(x):=\frac{1}{(4\pi)^{\frac{d}{2}}}
    \int_0^\infty 
    \chi\Bigl(R^{-\frac12}x\Bigr) e^{-R}
    \frac{dR}{R^{\frac{d}{2}}},
    \qquad
    \chi(x):=e^{-\frac{\abs{x}^2}{4}},
    \qquad
    x \in \R^d \setminus \{0\}.
\end{equation}
In dimension $d\geq 3$, a basic estimate (see Lemma \ref{sec2:lemma}) involving only localization in $x$ gives that $\norm{\nabla_x\phi}_{L^\infty}=\mathcal{O}(t^{-d})$. This turns out to be sufficient also for nonlinear arguments, globally in time.

In dimension $d=2$, a more refined analysis is necessary. To capture the leading order behavior of \eqref{eq:rho_intro}, we follow \cite{IPWW2020} and introduce the norm
\begin{equation}
\label{znorm_intro}
    \norm{h}_Z:= \norm{h}_{L^\infty_v L^2_x},
\end{equation}
so that
\begin{equation*}
\norm{\rho_h(t)}_{L^\infty(\R^2)}\lesssim t^{-2}\norm{h}_Z^2+o(t^{-2}).
\end{equation*}
The corresponding sharp estimates for $\nabla_x\phi_h$ are direct, but not of particular importance for the nonlinear bootstrap, so instead we focus on decay estimates that give the almost optimal decay rate under slightly less stringent assumptions: in Proposition \ref{prop:decay2} we show that for $\delta>0$ there holds that
\begin{equation*}
 \norm{\partial^\alpha_x \nabla_x\phi(x,t)}_{L^\infty_x(\R^2)}\lesssim_{\delta} \ip{t}^{-2-|\alpha|+\delta}\min\left\{\norm{h}_{L^2_xH^{|\alpha|+1}_v}^2,\ip{t}^{-1}\norm{h}_{L^2_xH^{|\alpha|+2}_v}^2\right\},\quad\alpha\in\N_0^2. 
\end{equation*}

In dimension $d=1$, we work with analytic norms but are not able to obtain estimates that allow for sharp linear decay (see also the discussion in Section \ref{sec:intro_nonlin} below), and instead guarantee bounds of the order $\norm{(t\partial_x)^k\partial_x\phi}_{L^\infty}=\mathcal{O}(t^{-\frac32})$, $k\in\N_0$ -- see Lemma \ref{lem:1d_decay}.
 
\subsubsection{Nonlinear propagation of localization and regularity}\label{sec:intro_nonlin}
The local well-posedness of \eqref{eq:gamma} in the topologies used in this article is classical, and we thus focus here on the bootstrap arguments used to establish nonlinear control. We will thus assume that we are given a solution $\gamma$ to \eqref{eq:gamma} on at time interval $[0,T]$ that satisfies certain smallness, regularity and localization assumptions, and will subsequently show that the corresponding bounds can be improved, which by continuity implies that the initial assumptions on the solution hold on a longer time interval. As mentioned before, the crucial step hereby is to obtain suitable decay bounds for the force field (and its derivatives).

For $d\geq 3$, the dispersive decay is rapid enough that a straightforward bootstrap involving on the simple topology 
 $\norm{\ip{x}^m \gamma(t)}_{L^\infty_{x,v}(\R^d \times \R^d)} + \norm{\nabla_{x,v} \gamma}_{L^\infty_{x,v}(\R^d \times \R^d)}  $
can be closed globally in time -- see Lemma \ref{sec2:bootstrap}.

For $d=2$, the main difficulty is to close estimates for the velocity regularity. For this, we lean on the approach developed in \cite{IPWW2020} and establish a hierarchy of bounds where the $L^2_xH^j_v$, $j=2,3$, energies of $\gamma$ grow slowly over time, whereas several lower order quantities including the aforementioned $Z$-norm, an $x$-moment and $x$-derivatives remain uniformly bounded, i.e.\ we show in Proposition \ref{sec3:boot} that the following bounds can be bootstrapped:
\begin{equation*}
\|\gamma\|_{L^2_x H^3_v}\le \varepsilon \ip{t}^\delta,\quad \|\gamma\|_Z + \|\gamma \|_{L^2_v H^3_x} + \|x \gamma \|_{L^2_{x,v}}\le \varepsilon. 
\end{equation*}
The latter bounds follow essentially from the corresponding linear decay estimates, since $x$-derivatives of $\phi$ decay at faster rates. For  velocity regularity $\norm{\gamma(t)}_{L^2_xH^j_v}$, $j=2,3$, we make use of the symmetric structure of the energy estimates to carry out a multilinear analysis, relying also on the precise structure of the screened kernel. A crucial ingredient hereby is a triangular structure of the system of derivatives in \eqref{eq:gamma}: for regularity in $v$, the highest order terms in $t$ are forced by derivatives in $x$, which in turn do not grow over time. The resulting analysis requires some care, and is detailed in Proposition \ref{sec3:boot} and its proof. We sketch here a simplified picture of this approach: let us consider a prototypical term that arises in the propagation of two velocity derivatives $\partial_{v_j}^2\gamma$ in \eqref{eq:gamma} and reads\footnote{This corresponds to the contribution $\nabla_x\phi(x+tv,t)\cdot(-t\nabla_x\gamma)$ on the right hand side side of \eqref{eq:gamma}, with the derivative $\partial_{v_j}^2$ landing on $\nabla_x\phi$, paired with $\partial_{v_j}^2\gamma$.}
\begin{equation}\label{eq:sample}
\begin{aligned}
 &\ip{\partial_{v_j}^2(\nabla_x\phi(x+tv,t))t\nabla_x\gamma,\partial_{v_j}^2\gamma}_{L^2_{x,v}}=t^3\ip{\partial_{x^j}^2\nabla_x\phi(x+tv,t)\cdot\nabla_x\gamma,\partial_{v_j}^2\gamma}_{L^2_{x,v}} \\
 &\;=t^3\iint_{x,v\in\R^2} \int_{y\in\R^2}\partial_{x^j}^2\nabla_x G_2(x+tv-y)\rho(y,t)dy\cdot\nabla_x\gamma(x,v,t)\partial_{v_j}^2\gamma(x,v,t)dxdv\\
 &\;=t^3\iint_{x,v\in\R^2} \iint_{y,w\in\R^2}\partial_{x^j}^2\nabla_x G_2(x+tv-y)\gamma^2(y-tw,w,t)\\
 &\hspace{5cm}\cdot\nabla_x\gamma(x,v,t)\partial_{v_j}^2\gamma(x,v,t)dydwdxdv\\
 &\;=-t^{-1}\iint_{x,v\in\R^2} \iint_{y,w\in\R^2}\!\!\!\partial_{y^j}^2\nabla_y G_2(x-y)\gamma^2\Bigl(a,\frac{y-a}{t},t\Bigr)\\
 &\hspace{5cm}\cdot\nabla_x\gamma\Bigl(b,\frac{x-b}{t},t\Bigr)\partial_{v_j}^2\gamma\Bigl(b,\frac{x-b}{t},t\Bigr)dydadxdb,
\end{aligned} 
\end{equation}
where in the last line we used the dispersive change of variables $w\mapsto y-tw=:a$, $v\mapsto x-tv=:b$. One may now integrate by parts the derivatives on $G_2$ to obtain further decay in time, albeit at the cost of regularity in $v$. Notice moreover that in order to close the estimates, the highest order velocity derivatives in $\gamma$ have to be estimated in the corresponding $L^2$-based spaces, while for lower orders we have more flexibility and can in particular make use of the boundedness of $\norm{\gamma(t)}_Z$. Amongst others, two integration by parts in $y_j$ in \eqref{eq:sample} give rise to the contribution
\begin{equation*}
\begin{aligned}
  &\Biggl\vert t^{-3}\iint_{x,v\in\R^2} \iint_{y,w\in\R^2}\!\!\!\nabla_y G_2(x-y)\gamma\Bigl(a,\frac{y-a}{t},t\Bigr)\partial_{v_j}^2\gamma\Bigl(a,\frac{y-a}{t},t\Bigr)\\
 &\hspace{5cm}\cdot\nabla_x\gamma\Bigl(b,\frac{x-b}{t},t\Bigr)\partial_{v_j}^2\gamma\Bigl(b,\frac{x-b}{t},t\Bigr)dydadxdb\Biggr\vert\\
 &\qquad \lesssim t^{-1}\norm{\gamma}_Z\norm{\partial_{v_j}^2\gamma}_{L^2_{x,v}}\norm{\nabla_x\gamma}_Z\norm{\partial_{v_j}^2\gamma}_{L^2_{x,v}} +l.o.t.\\
 &\qquad\lesssim t^{-1}\eps^2 \norm{\partial_{v_j}^2\gamma}_{L^2_{x,v}}^2 +l.o.t.,
\end{aligned} 
\end{equation*}
where in the last line we have used the bootstrap assumptions and denoted lower order terms as l.o.t. Via Gr\"onwall's inequality this leads to an acceptable contribution.

For $d=1$, decay of $\norm{\partial_x\phi}_{L^\infty}$ at the sharp rate $t^{-2}$ requires the control of $\norm{\partial_v\gamma}_Z$, as is clear from the linear analysis. However, this leads to a double derivative loss, since $Z$ is not an energy norm, and precludes a direct nonlinear analysis even in an analytic functional setting. To avoid this issue, we work exclusively with energy type norms of analytic regularity and time-dependent analyticity parameter, as in \eqref{eq:intro_ana_norms}.
These are not sufficient to obtain sharp decay estimates, and as a consequence we are not able to understand the asymptotic behavior of solutions. Instead, we close a nonlinear bootstrap (see Proposition \ref{sec7:boot}) on a long time span $T$, quantified in terms of the size $\eps$ of the initial data as $T\sim \mathcal{O}(\eps^{-4})$.

\section{Free scattering in $d\ge 3$}
\label{sec2}
In this section we establish Theorem \ref{intro:thm1} for $d\geq 3$, showing the stability of vacuum in \eqref{eq:scrVP} under sufficiently small, smooth and localized perturbations. As discussed in the introduction (see Section \ref{sec:intro_proofs}), we do this by studying the evolution of the profile $\gamma(x,v,t):=\mu(x+tv,v,t)$, which satisfies the equation \eqref{eq:gamma} with initial data $\gamma(x,v,0)=\mu(x,v,0)=\mu_0(x,v)$. In a first step, in Section \ref{sec:3d_lin} we establish the relevant decay estimates for $\nabla_x\phi$. These are then used in Section \ref{sec:3d_nonlin} to establish a bootstrap that globally controls solutions and allows us to deduce the asymptotic behavior, a free scattering dynamic.

\subsection{Linear analysis: dispersive decay estimates}\label{sec:3d_lin}
We recall from Section \ref{sec:intro_lin} that the Green function for $(1-\Delta)$ on $\R^d$ is given by 
\begin{equation*}
    G_d(x)=\frac{1}{(4\pi)^{\frac{d}{2}}}
    \int_0^\infty 
    \chi\Bigl(R^{-\frac12}x\Bigr) e^{-R}
    \frac{dR}{R^{\frac{d}{2}}},
    \qquad
    \chi(x)=e^{-\frac{\abs{x}^2}{4}},
    \qquad
    x \in \R^d \setminus \{0\},
\end{equation*}
and will use the notation
\begin{equation}
\label{sec2:greenconv}
    \rho_h(x,t):=\int h^2(x-tw,w)dw,\qquad \phi_h(x,t):=(1-\Delta)^{-1}\rho_h(x,t)=(G_d*\rho_h)(x,t)
\end{equation}
for $h:\R_x^d\times\R_v^d  \to \R$.
We establish the following linear decay estimates:
\begin{lemma}[Linear Decay in $d\ge3$] 
\label{sec2:lemma}
For $m>d/2$ there holds that
    \begin{equation}
    \label{in1:d3}
        \norm{\nabla_x \phi_h(x,t)}_{L^\infty_x(\R^d)}
        \lesssim 
        \ip{t}^{-d} 
        \norm{\ip{x}^m h}_{L^\infty_{x,v}(\R^d \times \R^d)}^2,
    \end{equation} 
while for $m>d$
    \begin{align}   
        \norm{\nabla_x \phi_h}_{L^\infty_x}&\lesssim \ip{t}^{-d-1} \norm{\ip{x}^m h}_{L^\infty_{x,v}} \norm{\nabla_v h}_{L^\infty_{x,v}},\label{in1.1:d3}\\
        \|\partial_{x_k} \nabla_x \phi_h(x,t)\|_{L^\infty_x} &\lesssim \ip{t}^{-d-1}\|\ip{x}^m h\|_{L^\infty_{x,v}} \|\nabla_v h\|_{L^\infty_{x,v}},\qquad k\in\{1,\ldots, d\}. \label{in2:d3}
    \end{align}    
\end{lemma}
\begin{proof}
We assume without loss of generality that $t\geq 1$. From \eqref{green} we obtain after a change of variables $w\mapsto x-y-tw=:a$ that
\begin{equation}
\label{eq:3dlinpf_1}
\begin{aligned}
    \nabla_x\phi_h (x,t)
    &=
    \int_{\R^d_y} \nabla_xG_d(x-y) \rho_h(y,t) dy
    = \iint_{\R^d_{y,w}}
    \nabla_yG_d(y)h^2(x-y-tw,w) dy dw
    \\
    &=t^{-d}\frac{1}{(4\pi)^{\frac{d}{2}}}\int_0^\infty \frac{e^{-R}}{R^{\frac{d+1}{2}}}
    \iint_{\R^d_{y,a}}
    \nabla\chi\Big(R^{-\frac{1}{2}}y\Big)h^2\Big(a,\frac{x-y-a}{t}\Big) 
    dy da dR.
\end{aligned}
\end{equation}
It follows that for a given integer $m>d/2$ there holds that
\begin{align*}
    \|\nabla_x\phi_h (x,t)\|_{L^\infty_x}
    &\le t^{-d}
    \|\ip{x}^m h\|^2_{L^\infty_{x,v}}\int_0^\infty  \frac{e^{-R}}{(4\pi)^{\frac{d}{2}} R^{\frac{d+1}{2}}}   
    \iint_{\R^d_{y,a}}
    \nabla\chi\Big(R^{-\frac{1}{2}}y\Big) \frac{dy da}{\ip{a}^{2m}}dR \\
    &\lesssim  t^{-d}\|\ip{x}^m h\|^2_{L^\infty_{x,v}} \int_0^\infty \frac{ e^{-R}}{ R^{\frac12}}dR,
\end{align*}
where in the first inequality we multiplied and divided by $\ip{a}^{2m}$, taking the $L^\infty$ norm of the $m$-th spatial moment of $h$. This yields the bound in \eqref{in1:d3}. For \eqref{in1.1:d3} we observe that an integration by parts in \eqref{eq:3dlinpf_1} shows that
\begin{equation}\label{eq:nablaphiibp}
\begin{aligned}
    \nabla_x\phi_h (x,t)
    &=\frac{t^{-d-1}}{(4\pi)^{\frac{d}{2}}}\int_0^\infty  
    \frac{e^{-R}}{R^{\frac{d}{2}}}
    \iint_{\R^d_{z,a}}
    \chi\Big(R^{-\frac{1}{2}}z\Big)
    2 h\Big(a,\frac{x-z-a}{t}\Big) \ \nabla_v h\Big(a,\frac{x-z-a}{t}\Big) 
    dz da dR,
\end{aligned}
\end{equation}
and thus for $m>d$
\begin{equation*}
    \norm{\nabla_x \phi_h}_{L^\infty_x}\lesssim 
    \ip{t}^{-d-1} \norm{\ip{x}^m h}_{L^\infty_{x,v}} \norm{\nabla_v h}_{L^\infty_{x,v}}.
\end{equation*}
Regarding \eqref{in2:d3}, an analogous computation gives that for $k \in \{1,\dots, d\}$
\begin{equation*}
\begin{aligned}
    \partial_{x_k} \nabla_{x}\phi_h (x,t)
    &=\frac{t^{-d-1}}{(4\pi)^{\frac{d}{2}}}\int_0^\infty 
         \frac{e^{-R}}{R^{\frac{d+1}{2}}}\\
    &\qquad \cdot 
    \iint_{\R^d_{z,a}}
    \nabla\chi\Big(R^{-\frac{1}{2}}z\Big)
    2 h\Big(a,\frac{x-z-a}{t}\Big) \partial_{v_k} h\Big(a,\frac{x-z-a}{t}\Big) 
    dz dadR,
\end{aligned}
\end{equation*}
and the claim follows as above.
\end{proof}

\subsection{Nonlinear analysis: bootstrap estimates and free scattering}\label{sec:3d_nonlin}
Next we establish the nonlinear bootstrap required that will allow us to understand the asymptotic behavior.

\begin{lemma}[Bootstrap estimate in $d \ge 3$]
\label{sec2:bootstrap}

There exists $\eps_0>0$ such that if $\gamma$ is a solution to \eqref{eq:gamma} on $[0,T]$ satisfying for $0<\eps<\eps_0$ the initial data bounds \eqref{eq:assump_3d} and
\begin{equation}
    \label{sec2:in1}
    \|\ip{x}^m \gamma(t)\|_{L^\infty_{x,v}(\R^d \times \R^d)} +\|\nabla_{x,v} \gamma(t)\|_{L^\infty_{x,v}(\R^d \times \R^d)}\le 4\varepsilon, 
    \end{equation}
then in fact we have the improved bounds
\begin{equation}
    \label{sec2:in3}
    \|\ip{x}^m \gamma(t) \|_{L^\infty_{x,v}(\R^d\times \R^d)} +
    \|\nabla_{x,v} \gamma(t)\|_{L^\infty_{x,v}(\R^d \times \R^d)}\le 2\varepsilon.
    \end{equation}
\end{lemma}

\begin{proof}
Under the assumptions \eqref{sec2:in1} we observe that by Lemma \ref{sec2:lemma} we have
    \begin{equation}
    \label{sec2:in2}
            \|\nabla_x \phi(x,t)\|_{L^\infty_x(\R^d)} \lesssim \ip{t}^{-d} \varepsilon^2,
            \quad \quad
            \|\partial_{x_k} \nabla_x \phi(x,t)\|_{L^\infty_x(\R^d)} \lesssim \ip{t}^{-d-1} \varepsilon^2,\quad 1\leq k\leq d.
    \end{equation}
To prove \eqref{sec2:in3}, we compute from \eqref{eq:gamma} that
\begin{align*}
    \partial_t\{x_j \gamma\}&-\nabla_x \phi(x+tv,t)\cdot \{\nabla_v-t\nabla_x\}\{x_j\gamma\}=
    t\partial_{x_j}\phi(x+tv,t)\gamma,\\
    \partial_t \{\partial_{x_j}\gamma\} &- \nabla_x \phi(x+tv,t)\cdot \{\nabla_v-t\nabla_x\}\{\partial_{x_j}\gamma\} =
    \partial_{x_j}\nabla_x \phi(x+tv,t) \cdot \{\nabla_v -t\nabla_x\}\gamma,\\
    \partial_t \{\partial_{v_j}\gamma\} &- \nabla_x \phi(x+tv,t)\cdot \{\nabla_v-t\nabla_x\}\{\partial_{v_j}\gamma\}  =
    t\partial_{x_j}\nabla_x \phi(x+tv,t) \cdot \{\nabla_v -t\nabla_x\}\gamma.
\end{align*}
The left-hand side of the above equations is conservative, and in particular preserves all $L^p$ norms.\footnote{To put it differently, observe that the vector field
    $V(x,v,t):=(-t\nabla_x\phi(x+tv,t), \nabla_x\phi(x+tv,t))$ 
is divergence-free.} By \eqref{sec2:in1} and \eqref{sec2:in2} it follows that for $d\geq 3$
\begin{equation}
\label{sec2:est1}
\begin{aligned}
    \|x_j \gamma(t)\|_{L_{x,v}^\infty}&\le   \|x_j \gamma(0)\|_{L_{x,v}^\infty} + \int_0^t s\|\nabla_x\phi(s)\|_{L^\infty_x} \|\gamma(s)\|_{L^\infty_{x,v}}ds\\
    &\le \varepsilon + 4\varepsilon^3\int_0^t \frac{s}{\ip{s}^{d}}ds \\
    &\le \varepsilon + C\varepsilon^3.
\end{aligned}
\end{equation}
For $k\in\N$ we can reason inductively with
\begin{equation*}
    \partial_t\{x^k_j \gamma\} -\nabla_x\phi(x+tv,t)\cdot \{\nabla_v - t\nabla_x\}\{x_j^k\gamma\}=tk \partial_{x_j}\phi(x+tv,t) \{x_j^{k-1}\gamma\}
\end{equation*}
and with bounds as in \eqref{sec2:est1} to obtain the estimate for the higher moments in \eqref{sec2:in3}.

Analogously, thanks to \eqref{sec2:in2} we obtain uniform estimates for the gradients. We give the details for the more difficult case of the velocity gradient, since its equation contains an extra growing factor in time: 
by \eqref{sec2:in1} and
\eqref{sec2:in2} we have for $1\leq j\leq d$ that
\begin{align*}
\|\partial_{v_j}\gamma(t)\|_{L^\infty_{x,v}}&\le \|\partial_{v_j}\gamma(0)\|_{L^\infty_{x,v}} + \int_0^t s\|\partial_{x_j}\nabla_x\phi\|_{L^\infty_x}(\|\nabla_v \gamma\|_{L^\infty_{x,v}}+s \|\nabla_x \gamma\|_{L^\infty_{x,v}}) ds \\
&\le \eps + 8\int_0^t \frac{\varepsilon^3}{\ip{s}^{d-1}}ds\le 
\eps + C \varepsilon^3
\end{align*}
which concludes the proof.
\end{proof}

\begin{remark}\label{rem:sharpened3d}
 As the above proof shows, spatial moments alone suffice to close a global bootstrap. With additional regularity, by \eqref{in1.1:d3} one obtains sharp decay of $\nabla_x\phi$, and uniqueness in a strong topology.
\end{remark}

This bootstrap allows us to construct a global solution satisfying the bounds \eqref{sec2:in3}, for which we can now deduce the asymptotic behavior:
\begin{theorem}[Free scattering in $d\ge 3$]
\label{sec2:main"}
    Given an integer $d\ge 3$, there exists $\eps_0>0$ such that if for some $m>d$
    \begin{equation*}
    \|\ip{x}^m \mu_0\|_{L^\infty_{x,v}(\R^d \times \R^d)}+ \|\nabla_{x,v}\mu_0\|_{L^\infty_{x,v}(\R^d \times \R^d)} \leq \eps \le \varepsilon_0,
    \end{equation*}
    then there exists a unique solution $\gamma\in C_t(\R,W_{x,v}^{1,\infty}(\R^d\times\R^d))$ of \eqref{eq:gamma} with initial data $\gamma(0)=\mu_0$ and
    which verifies
    \begin{equation*}
        \|\ip{x}^m \gamma(t)\|_{L^\infty_{x,v}(\R^d \times \R^d)} +
        \|\nabla_{x,v}\gamma(t)\|_{L^\infty_{x,v}(\R^d \times \R^d)} 
        \lesssim \eps,\qquad \norm{\nabla_x\phi(t)}_{L^\infty}\lesssim \ip{t}^{-d-1}\eps^2.
    \end{equation*}
    In particular $\mu(x,v,t)=\gamma(x-tv,v,t)$ exhibits free scattering: there exists an asymptotic distribution function $\gamma_\infty \in L^\infty_{x,v}$ such that
    \begin{equation}
    \label{eq:frscat3}
        \lim_{t \to \infty} \| \mu(x+tv,v,t) - \gamma_\infty(x,v)\|_{L^\infty_{x,v}}= 0.
    \end{equation}
\end{theorem}

\begin{proof}
   The existence of a unique global solution $\gamma$ of \eqref{eq:gamma} with the claimed properties follows from classical well-posedness theory and the bootstrap in Lemma \ref{sec2:bootstrap}.
    It remains to show the free scattering property \eqref{eq:frscat3}: Since the decay estimate \eqref{sec2:in2} holds, it follows directly from the equation \eqref{eq:gamma} that
    \begin{equation*}
        \|\partial_t \gamma\|_{L^\infty_{x,v}}\le \|\nabla_x \phi\|_{L^\infty_x} (\|\nabla_v \gamma\|_{L^\infty_{x,v}} + t \|\nabla_x \gamma\|_{L^\infty_{x,v}})\lesssim \frac{\eps^3}{\ip{t}^{d-1}}.
    \end{equation*}
    It follows that
    $\|\gamma(t)\|_{L^\infty_{x,v}}$ is Cauchy-continuous in time, and since $d\ge 3$ there exists $\gamma_\infty\in L^\infty_{x,v}$ such that $\gamma(t)$ converges uniformly to $\gamma_\infty$ as $t\to\infty$. Translated to the variable $\mu(x,v,t)=\gamma(x-tv,v,t)$, this proves \eqref{eq:frscat3}.
\end{proof}

\section{Free scattering in $d=2$}
\label{sec3}
In this section we establish Theorem \ref{intro:thm1} for $d=2$. The arguments proceed along broadly similar lines as in case $d\geq 3$, beginning with an analysis of linear decay in Section \ref{sec:2d_lin}, a nonlinear bootstrap in Section \ref{sec:2d_bootstrap} and finally the demonstration of free scattering of solutions in Section \ref{sec:2dfree_scat}, but require significant refinements. As discussed in the introduction, we rely on a combination of higher energy norms that are allowed to grow slowly in time with a lower order $Z$-norm that captures the sharp decay behavior and remains uniformly bounded.

We begin by recalling from \eqref{znorm_intro} the definition of the relevant norm
\begin{equation}
\label{znorm}
    \norm{h}_Z:= \norm{h}_{L^\infty_v L^2_x}.
\end{equation}
As will be clear below, for certain error terms the following auxiliary norm will be useful:
\begin{equation}
\label{zprime}
    \norm{h}_{Z'}
    :=
    \norm{h}_{Z} + \ip{t}^{-\frac14}\left(\norm{h}_{L^2_x H^2_v} + \norm{xh}_{L^2_{x,v}}\right).
\end{equation}
Its key feature is that
\begin{equation}\label{sec3:PL}
    \Big\| h\Big(a,\frac{x-a}{t}\Big) - h\Big(a, \frac{x}{t}\Big) \Big\|_{L^2_aL^\infty_x } \lesssim t^{-\frac{1}{3}} 
     \left(\|h\|_{L^2_x H^2_v}+ \|x h\|_{L^2_{x,v}} \right)\lesssim \norm{h}_{Z'},
\end{equation}
which can be shown using a standard Littlewood-Paley decomposition in velocity, which (to fix notation) we write as
\begin{equation}\label{sec3:defLP}
 h=\sum_{M \in 2^{\Z}} h_M,\quad h_M(x,v):=\mathcal{F}^{-1}_{\eta \to v}\left[\psi\left(M^{-1}\eta\right) \mathcal{F}_{v \to \eta}h\right](x,v),
\end{equation}
for $h:\R^2\times\R^2\to\R$ -- see Appendix \ref{sec:appdx} for the details.

\subsection{Linear analysis}\label{sec:2d_lin}
Recall the notations \eqref{sec2:greenconv}.
\begin{proposition}[Linear Decay in $d=2$]
\label{prop:decay2}
We have that
\begin{equation}\label{eq:rho_bound}
    \norm{\rho_h(x,t)}_{L^\infty_x(\R^2)}\lesssim \ip{t}^{-2}\norm{h}^2_{Z'}.
 \end{equation}
 Moreover, for any $0<\delta<1$ and $|\alpha|\ge 0$ there holds that
 \begin{align}
 \norm{\partial^\alpha_x \nabla_x\phi_h(x,t)}_{L^\infty_x(\R^2)}&\lesssim_{\delta} \ip{t}^{-2-|\alpha|+\delta}\norm{h}_{L^2_xH^{|\alpha|+1}_v}^2,\label{in:lin3}\\
 \norm{\partial^\alpha_x \nabla_x\phi_h(x,t)}_{L^\infty_x(\R^2)}&\lesssim_{\delta} \ip{t}^{-3-|\alpha|+\delta}\norm{h}_{L^2_xH^{|\alpha|+2}_v}^2.\label{in:lin2}
  \end{align}
 Furthermore, for $|\alpha| \ge 2$ we have
\begin{equation}\label{in:lin5}
    \norm{\partial^\alpha_x \nabla_x\phi_h(x,t)}_{L^2_x} \lesssim \ip{t}^{-\frac12-|\alpha| }\|h\|^2_{L^2_x H^{|\alpha|}_v}.
\end{equation}
\end{proposition}

The proof proceeds similarly as in \cite[Section 2]{IPWW2020}.
\begin{proof}
The change of variables $w\mapsto x-tw=:a$ yields 
\begin{equation}
\label{prop1:est1}
\begin{aligned}
\rho_h(x,t)
&= \int_{\R^2}h^2(x-tw, w) dw
=t^{-2}\int_{\R^2}h^2\Big(a, \frac{x-a}{t}\Big)da\\
&
=t^{-2}\int_{\R^2}h^2\Big(a, \frac{x}{t}\Big) da + J(x,t),
\end{aligned}
\end{equation}
where
\begin{equation*}
    J(x,t):=t^{-2}\int_{\R^2}
    \Big\{ 
    h^2\Big(a, \frac{x-a}{t}\Big) - 
    h^2\Big(a, \frac{x}{t}\Big) 
    \Big\} da.
\end{equation*}
By construction, the leading order term in \eqref{prop1:est1} can be bounded using the $Z$-norm in \eqref{znorm}
\begin{equation*}
\int_{\R^2}h^2\Big(a, \frac{x}{t}\Big) da 
\le \norm{h(t)}_Z^2,
\end{equation*}
while the remainder $J(x,t)$ satisfies by \eqref{sec3:PL} that
\begin{align*}\label{sec3:J}
    |J(x,t)| 
    &\lesssim t^{-2}\|h\|_{L^2_xL^\infty_v }
    \Big\|h\Big(a,\frac{x-a}{t}\Big) - h\Big(a, \frac{x}{t}\Big)\Big\|_{L^2_aL^\infty_x }\lesssim t^{-2-\frac{1}{3}}\left(\|h\|^2_{L^2_x H^2_v}+ \|h\|_Z\|x h\|_{L^2_{x,v}}\right),
\end{align*}
since $\|h\|_{L^2_x L^\infty_v(\R^2)}\le \|h\|_{L^2_x H^2_v(\R^2)}$ 
by the Sobolev embedding. This gives \eqref{eq:rho_bound}.

To bound the force field and its derivatives, we use the Green function \eqref{green} of the screened Laplacian to write
\begin{equation}
\label{eq:nablaPhi}
\begin{aligned}
\nabla_x\phi_h (x,t)&=\nabla_x\iint_{y,w\in\R^2}G_2(x-y)h^2(y-tw,w)dydw\\
&= \frac{1}{4\pi}\int_0^\infty \Phi_R(x,t) e^{-R}\frac{dR}{R},
\end{aligned}
\end{equation}
where
\begin{equation*}
\Phi_R(x,t):=\iint_{\R^2_{y,w}}
\nabla_x \Big\{\chi\Big(R^{-\frac{1}{2}}(x-y)\Big)\Big\}h^2(y-tw,w)
dy dw.
\end{equation*}
Given a multi-index $\alpha\in\N_0^2$, we change variables as in \eqref{prop1:est1} and integrate by parts to obtain that
\begin{equation}
\label{sec3:phiR}
    \partial^\alpha_x\Phi_R(x,t)=
    t^{-2-|\alpha|}
     \sum_{\substack{\alpha_1+\alpha_2=\alpha \\ |\alpha_1|\le |\alpha_2|}}c_{\alpha_1, \alpha_2} \Phi_R^{\alpha_1, \alpha_2}(x,t),
     \end{equation}
     where
\begin{equation}\label{eq:Phi_loc}
     \Phi_R^{\alpha_1, \alpha_2}(x,t):=
    \iint_{\R^2_{z,a}}\nabla_z
    \Big\{\chi\Big(R^{-\frac{1}{2}}z\Big)\Big\}
    \partial_v^{\alpha_1}h\Big(a, \frac{x-z-a}{t}\Big) 
    \partial^{\alpha_2}_v h\Big(a,\frac{x-z-a}{t}\Big) dz da.
\end{equation}
We next establish \eqref{in:lin3}. If $|\alpha|=0$ we have by Hölder's inequality for $2<q<\infty$ that
\begin{equation*}
\Big|\Phi^{0, 0}_R(x,t)\Big|\le\Big\|\nabla_z\Big\{\chi\Big(R^{-\frac{1}{2}}z\Big)\Big\}\Big\|_{L^{q'}_z} \Big\|h\Big(a, \frac{x-z-a}{t}\Big)\Big\|^2_{L^2_aL^q_z},
\end{equation*}
and similarly if $|\alpha|>0$ we get
\begin{equation}
\label{sec3:fieldest}
\begin{aligned}
\Big|\Phi^{\alpha_1, \alpha_2}_R&(x,t)\Big|
 \le
\Big\|\nabla_z
\Big\{\chi\Big(R^{-\frac{1}{2}}z\Big)\Big\}\Big\|_{L^{q'}_z} 
\Big\|\partial^{\alpha_2}_v 
h\Big(a, \frac{x-z-a}{t}\Big)\Big\|_{L^2_aL^q_z } 
\Big\|
\partial_v^{\alpha_1}h\Big(a, \frac{x-z-a}{t}\Big)
\Big\|_{L^2_a L^\infty_z}.
\end{aligned}
\end{equation}
We thus obtain that
\begin{equation*}
\left|\Phi^{\alpha_1, \alpha_2}_R(x,t)\right|\lesssim_q t^{\frac{2}{q}}R^{\frac12-\frac{1}{q}} \|h\|^2_{L^2_x H^{|\alpha|+1}_v},
\end{equation*}
where we also used that by Sobolev embedding
and since $|\alpha_1|<|\alpha|$ there holds that
\begin{equation*}
\left\|\partial^{\alpha_2}_v 
h\right\|_{L^2_xL^q_v } 
\lesssim_q \|h\|_{L^2_xH^{|\alpha_2|+1}_v},
\qquad
\left\|\partial^{\alpha_1}_v 
h\right\|_{L^2_x L^\infty_v} \lesssim \|h\|_{L^2_xH^{|\alpha_1|+2}_v}\le \|h\|_{L^2_x H^{|\alpha|+1}_v}.
\end{equation*}
Recalling the expression in \eqref{eq:nablaPhi}, 
this gives the claimed bound
\begin{align*}
\norm{\partial_x^\alpha\nabla_x \phi_h(x,t)}_{L^\infty_x}&\lesssim_q t^{-2-|\alpha|+\frac{2}{q}}\|h\|^2_{L^2_x H^{|\alpha|+1}_v}\int_0^\infty R^{-\frac12-\frac{1}{q}}e^{-R}dR\lesssim_q  t^{-2-|\alpha|+\frac{2}{q}}\|h\|^2_{L^2_x H^{|\alpha|+1}_v}.
\end{align*}
The bound \eqref{in:lin2} follows as in \eqref{sec3:fieldest} upon also integrating by parts the gradient on $\chi(R^{-\frac{1}{2}}z)$ in \eqref{eq:Phi_loc}.

Finally, the $L^2$ estimate \eqref{in:lin5} is obtained
by interpolation. On the one hand, we bound the $L^1$ norm as
\begin{equation*}
\| \Phi^{\alpha_1,\alpha_2}_R(x,t)\|_{L^1_x}
\le\Big \| \nabla_z \{ \chi(R^{-\frac12}z)\}\Big \|_{L^1_z} 
\Big\| \partial^{\alpha_2}_v h\Big(a, \frac{x-z-a}{t}\Big) \Big\|_{L^2_aL^2_z}
\Big\|\partial_v^{\alpha_1}h\Big(a, \frac{x-z-a}{t}\Big)\Big\|_{L^2_aL^2_z},
\end{equation*}
Since $|\alpha_1|\leq |\alpha_2|\le |\alpha|$ we get that
\begin{equation*}
\| \Phi^{\alpha_1,\alpha_2}_R(x,t)\|_{L^1_x}
\lesssim t^2 R^{\frac12} \|h\|^2_{L^2_x H_v^{|\alpha|}} 
\end{equation*}
so that by \eqref{sec3:phiR}
\begin{equation}
\label{sec3:interL1}
 \left\|\partial_x^\alpha \Phi_R(x,t)\right\|_{L^1_x} \lesssim t^{-|\alpha|}R^{\frac12} \|h\|^2_{L^2_x H_v^{|\alpha|}}.
\end{equation}
On the other hand, assuming that $|\alpha|\ge 2$ we have by Hölder's inequality with $\frac{1}{p_i}:=\frac{|\alpha_i|}{2|\alpha|}$, $i=1,2$, that
\begin{align*}
\Bigl \|&\Phi^{\alpha_1, \alpha_2}_R(x,t)\Bigr\|_{L^\infty_x}
&\le \norm{\nabla_z\Big\{\chi\Big(R^{-\frac{1}{2}}z\Big)\Big\}}_{L^{2}_z} 
\Big\|\partial^{\alpha_1}_vh\Big(a, \frac{x-z-a}{t}\Big)\Big\|_{ L^2_aL^{p_1}_z}
\Big\|\partial^{\alpha_2}_vh\Big(a, \frac{x-z-a}{t}\Big)\Big\|_{ L^2_aL^{p_2}_z},
\end{align*}
which yields
\begin{equation*}
    \norm{\Phi^{\alpha_1, \alpha_2}_R(x,t)}_{L^\infty_x}  
    \lesssim t \|\partial_v^{\alpha_1}h\|_{L^2_x L^{p_1}_{v}}
     \|\partial_v^{\alpha_2}h\|_{L^2_x L^{p_2}_{v}}\lesssim t \|h\|^2_{L^2_x H^{|\alpha|}_v}
\end{equation*}
by the Sobolev embedding. From \eqref{sec3:phiR} we conclude that
\begin{equation}
\label{sec3:interLinf}
 \norm{\partial_x^\alpha \Phi_R(x,t)}_{L^\infty_x}\lesssim t^{-1-\abs{\alpha}}\|h\|^2_{L^2_x H^{|\alpha|}_v}.  
\end{equation}
Interpolating the estimates \eqref{sec3:interL1} and \eqref{sec3:interLinf} shows that
\begin{equation*}
\|\partial_x^\alpha \Phi_R(x,t) \|_{L^2_x}
\le
\|\partial_x^\alpha \Phi_R(x,t) \|^{\frac12}_{L^\infty_x}
\|\partial_x^\alpha \Phi_R(x,t) \|^{\frac12}_{L^1_x}
\lesssim t^{-|\alpha|-\frac12}R^{\frac14} \|h\|^2_{L^2_x H_v^{|\alpha|}},
\end{equation*}
from which integration over $R$ as above yields \eqref{in:lin5}.
\end{proof}

\subsection{Nonlinear bootstrap}\label{sec:2d_bootstrap}
We recall from \eqref{eq:assump_2d} that we shall assume that for $\eps>0$ to be determined, the initial datum for $\eqref{eq:gamma}$ verifies
\begin{equation}
\label{sec3:initialdat}
    \|x \gamma_0\|_{L^2_{x,v}} +\norm{\gamma_0}_Z+\norm{\nabla_x\gamma_0}_Z+\|\gamma_0\|_{H^3_{x,v}} \le \eps.
\end{equation}

\begin{proposition}[Bootstrap estimates in $d=2$]\label{sec3:boot}
There exists $\eps_0>0$ such that if $\gamma$ is a solution to \eqref{eq:gamma} on $[0,T]$ satisfying for $0<\eps<\eps_0$ the initial data bounds \eqref{sec3:initialdat} and the bootstrap hypotheses
\begin{equation}
\label{boot:hyp}
\begin{aligned}
\|\gamma\|_Z + \|\nabla_x\gamma\|_Z +\|\gamma \|_{L^2_v H^3_x} + \|x \gamma \|_{L^2_{x,v}}+
\|\gamma\|_{L^2_x H^3_v}\le \varepsilon \ip{t}^\delta,
\end{aligned}
\end{equation}
for some $0<\delta\ll 1$, then the following improved estimates hold for universal constants $c_1,c_2>0$:
\begin{align}
\|\gamma\|_Z + \|\nabla_x \gamma\|_Z&\le  \varepsilon + c_1 \varepsilon^3, \label{boot:in1}\\
\|x\gamma\|_{L^2_{x,v}}
+\|\gamma\|_{L^2_v H^3_x}
+\|\gamma\|_{H^2_x H^1_v}
&\le \varepsilon + c_1 \varepsilon^3, \label{boot:in2}  \\
\|\gamma\|_{L^2_x H^2_v}&\le \varepsilon \ip{t}^{c_1\varepsilon^2}, \label{boot:in3}\\
\|\gamma\|_{L^2_x H^3_v}&\le c_2\varepsilon \ip{t}^{c_2\varepsilon^2}. \label{boot:in4}
\end{align}
\end{proposition}

The bounds on the $Z$-norm, spatial regularity and spatial moments are essentially direct consequences of the above decay estimates. The control of velocity regularity is the most delicate part of the proof and requires a multilinear analysis. 

\begin{proof}
As a direct consequence of the  bootstrap hypotheses in \eqref{boot:hyp}, by Proposition \ref{prop:decay2} we have the decay estimates
\begin{align}
\|\partial^\alpha_x \nabla_x \phi\|_{L^\infty_x} &\lesssim_{\delta}\ip{t}^{-3-|\alpha|+3\delta}\varepsilon^2, \quad |\alpha| \le 1 \label{boot:dec1}\\
\|\partial^{\alpha}_x \nabla_x \phi\|_{L^\infty_x} &\lesssim_\delta \ip{t}^{-4+3\delta}\varepsilon^2, \quad |\alpha|=2 \label{boot:dec2}\\
\|\partial^{\alpha}_x \nabla_x \phi\|_{L^2_x} &\lesssim \ip{t}^{-\frac72+2\delta}\varepsilon^2, \quad |\alpha|=3. \label{boot:dec3}
\end{align}
Moreover, by \eqref{sec3:initialdat} and the conservation of the $L^p$ norms of $\gamma$, it is immediate that $\|\gamma(t)\|_{L^2_{x,v}}\le \eps$. We prove the remaining bounds in \eqref{boot:in1}--\eqref{boot:in4} step by step.

\textbf{Propagation of the $Z$-norms in \eqref{boot:in1}.}

From \eqref{eq:gamma}, integrating by parts, we have
\begin{equation*}
    \frac12 \frac{d}{dt} \int_{\R^2_x} \gamma^2(x,v,t) dx =
    \int \nabla_x \phi(x+tv,t) \cdot \nabla_v \gamma(x,v,t) \gamma(x,v,t) dx
    + \frac{t}{2}\int \Delta_x \phi(x+tv,t) \gamma^2 dx.
\end{equation*}

Integrating in time the previous identity and using Cauchy-Schwarz, we get
\begin{equation*}
\begin{aligned}
    \int_{\R^2_x} \gamma^2(x,v,t) dx 
    &\le
    \int_{\R^2_x} \gamma^2_0(x,v)dx + 2\int_0^t 
     \norm{\nabla_x \phi(s)}_{L^\infty_x} 
     \norm{\nabla_v \gamma(x,v,s)}_{L^2_x} 
     \norm{\gamma(x,v,s)}_{L^2_x}ds\\
     &\qquad+ \int_0^t s \norm{\Delta_x \phi(s)}_{L^\infty_x} \norm{\gamma(x,v,s)}^2_{L^2_x} ds,
\end{aligned}
\end{equation*}
and hence
\begin{equation*}
    \| \gamma(t)\|^2_Z \le \|\gamma(0)\|^2_Z + \int_0^t
    \|\nabla_x \phi(s)\|_{L^\infty_x} \|\nabla_v \gamma(s) \|_Z \|\gamma(s)\|_Z
    ds
    +
    \int_0^t s  \|\Delta_x \phi(s)\|_{L^\infty_x} \|\gamma(s)\|^2_Z ds.
\end{equation*}
By Grönwall's inequality, the decay estimate \eqref{boot:dec1} and the Sobolev embeddings
\begin{equation*}
 \|\gamma\|_Z\lesssim \|\gamma\|_{L^2_x H^2_v}, \quad \|\nabla_v \gamma \|_Z \lesssim \|\gamma\|_{L^2_x H^3_v},
\end{equation*}
this implies the claimed bound for the $Z$-norm of $\gamma$, since $\norm{\gamma_0}_Z\leq \eps$ by assumption \eqref{sec3:initialdat}.

Similarly, since
\begin{equation*}
    \partial_t (\partial_{x_j} \gamma) = \nabla_x \phi(x+tv,t) \cdot \{\nabla_v -t \nabla_x\}(\partial_{x_j} \gamma)
    +\partial_{x_j} \nabla_x \phi(x+tv,t) \cdot \{\nabla_v -t \nabla_x\}\gamma,
\end{equation*}
we have
\begin{equation}
\label{sec3:partialgamma}
\begin{aligned}
    \frac12 \frac{d}{dt} \int_{\R^2_x} (\partial_{x_j}\gamma)^2 dx
    &=
    \int_{\R^2_x} \nabla_x \phi(x+tv,t)\cdot\partial_{x_j}\nabla_v \gamma \partial_{x_j}\gamma dx
    + \frac{t}{2}\int_{\R^2_x} \Delta_x \phi(x+tv,t) (\partial_{x_j}\gamma)^2 dx\\
    &\qquad +\int_{\R^2_x} \nabla_x \partial_{x_j}\phi(x+tv,t) \cdot \nabla_v \gamma \partial_{x_j}\gamma dx 
    -t \int_{\R^2_x} \nabla_x \partial_{x_j}\phi(x+tv,t) \cdot \nabla_x \gamma \partial_{x_j}\gamma dx\\
    &=:T_1+T_2+T_3+T_4.
\end{aligned}
\end{equation}
To estimate the last three terms $T_2, T_3, T_4$, a direct estimate suffices and gives
\begin{equation}
\label{boundT24}
\begin{aligned}
    |T_2 + T_3 + T_4| &\lesssim t\|\Delta_x \phi\|_{L^\infty_x}
    \|\partial_{x_j}\gamma\|_Z^2 + \|\partial_{x_j}\nabla \phi\|_{L^\infty_x}\|\nabla_v \gamma\|_Z\|\partial_{x_j}\gamma\|_Z \\
    &\qquad+ t \|\partial_{x_j}\nabla \phi\|_{L^\infty_x}\|\nabla_x \gamma\|_Z \|\partial_{x_j}\gamma\|_Z.
\end{aligned}
\end{equation}
To bound the first term $T_1$ in \eqref{sec3:partialgamma}, we slightly refine these arguments and use a Littlewood-Paley decomposition to write
\begin{equation*}
   T_1 = \sum_{C_1<C_2}T_1^{C_1,C_2} + \sum_{C_1 \ge C_2}T_1^{C_1,C_2}, 
\end{equation*}
where
\begin{equation*}
  T_1^{C_1,C_2}:= \int_{\R^2_x}\nabla_x \phi(x+tv,t) \cdot \partial_{x_j} \nabla_v \gamma_{C_1} \partial_{x_j}\gamma_{C_2} dx  
\end{equation*}
and we have used the Littlewood-Paley notation in \eqref{sec3:defLP}. For the first sum above it follows with the Bernstein inequalities that
\begin{equation}\label{sec3:estPL1}
\begin{aligned}
 \sum_{C_1<C_2}|T_1^{C_1,C_2}|&\le 
 \sum_{C_1<C_2}\|\nabla_x\phi\|_{L^\infty_x}\|\partial_{x_j}\nabla_v \gamma_{C_1}\|_{L^\infty_v L^2_x}\|\partial_{x_j} \gamma_{C_2}\|_{L^\infty_v L^2_x}\\
 &\le \sum_{C_1<C_2}\|\nabla_x\phi\|_{L^\infty_x}C_1\|\partial_{x_j}\nabla_v \gamma_{C_1}\|_{ L^2_{x,v}}C_2\|\partial_{x_j} \gamma_{C_2}\|_{ L^2_{x,v}}\\
 &\leq \|\nabla_x\phi\|_{L^\infty_x} \|\gamma\|_{H^1_xH^1_v}
 \|\gamma\|_{H^1_x H^2_v}.
\end{aligned}
\end{equation}
If $C_1 \ge C_2$ we integrate by parts to obtain that
\begin{equation}
    \label{sec3:intpartsPL}
\begin{aligned}
T_1^{C_1,C_2}&= -\int_{\R^2_x}\partial_{x_j}\nabla_x \phi(x+tv,t) \cdot \nabla_v \gamma_{C_1} \partial_{x_j}\gamma_{C_2} dx
-\int_{\R^2_x}\nabla_x \phi(x+tv,t) \cdot \nabla_v \gamma_{C_1} \partial^2_{x_j}\gamma_{C_2} dx\\
&=:T_{1,1}2^{C_1, C_2} + T_{1,2}^{C_1, C_2}.
\end{aligned}
\end{equation}
Reasoning as before, the first term in \eqref{sec3:intpartsPL} leads to
\begin{align*}
 \sum_{C_1\ge C_2}|T_{1,1}^{C_1,C_2}|&\le 
 \sum_{C_1\ge C_2}\|\partial_{x_j}\nabla_x\phi\|_{L^\infty_x}\|\nabla_v \gamma_{C_1}\|_{L^\infty_v L^2_x}\|\partial_{x_j} \gamma_{C_2}\|_{L^\infty_v L^2_x}\\
 &\le \sum_{C_1 \ge C_2}\|\partial_{x_j}\nabla_x\phi\|_{L^\infty_x}C_1\|\nabla_v \gamma_{C_1}\|_{ L^2_{x,v}}C_2\|\partial_{x_j} \gamma_{C_2}\|_{ L^2_{x,v}}\\
 &\le \|\partial_{x_j}\nabla_x\phi\|_{L^\infty_x} \sum_{C_1 \ge C_2}
 \Big(\frac{C_2}{C_1}\Big) \|\gamma_{C_1}\|_{L^2_x H^3_v}
 \|\gamma_{C_2}\|_{H^1_x L^2_v},
\end{align*}
while the second term is bounded similarly as
\begin{align*}
 \sum_{C_1\ge C_2}|T_{1,2}^{C_1,C_2}|&\le 
 \sum_{C_1\ge C_2}\|\nabla_x\phi\|_{L^\infty_x}\|\nabla_v \gamma_{C_1}\|_{L^\infty_v L^2_x}\|\partial^2_{x_j} \gamma_{C_2}\|_{L^\infty_v L^2_x}\\
 &\le \sum_{C_1 \ge C_2}\|\nabla_x\phi\|_{L^\infty_x}C_1\|\nabla_v \gamma_{C_1}\|_{ L^2_{x,v}}C_2\|\partial^2_{x_j} \gamma_{C_2}\|_{ L^2_{x,v}}\\
 &\le \|\partial_{x_j}\nabla_x\phi\|_{L^\infty_x} \sum_{C_1 \ge C_2}
 \Big(\frac{C_2}{C_1}\Big) \|\gamma_{C_1}\|_{L^2_x H^3_v}
 \|\gamma_{C_2}\|_{H^2_x L^2_v}.
\end{align*}
With \eqref{sec3:intpartsPL} it follows that
\begin{equation}
    \label{sec3:estPL2}
\sum_{C_1 \ge C_2}|T_1^{C_1,C_2}|
 \lesssim
\|\partial_{x_j}\nabla_x\phi\|_{L^\infty_x} \|\gamma\|_{H^2_x L^2_v}
 \|\gamma\|_{L^2_x H^3_v}.
\end{equation}
From \eqref{sec3:estPL1} and \eqref{sec3:estPL2} we thus have
\begin{equation}
    \label{boundT1}
    |T_1| \lesssim 
    \|\nabla_x\phi\|_{L^\infty_x} \|\gamma\|_{H^1_xH^1_v}
 \|\gamma\|_{H^1_x H^2_v}+\|\partial_{x_j}\nabla_x\phi\|_{L^\infty_x} \|\gamma\|_{H^2_x L^2_v}
 \|\gamma\|_{L^2_x H^3_v}.
\end{equation}

Collecting the estimates in \eqref{boundT1} and \eqref{boundT24}, integrating \eqref{sec3:partialgamma} in time and using the decay estimate \eqref{boot:dec1} and the bootstrap assumptions in \eqref{boot:hyp}, with the assumption on the initial data in \eqref{sec3:initialdat} we obtain the boundedness of $\|\nabla_x \gamma\|_Z$ stated in \eqref{boot:in1}.

We highlight that in particular it follows with the boundedness of $\norm{\gamma}_Z+\norm{\nabla_x\gamma}_Z$ and the bootstrap assumptions \eqref{boot:hyp} that the $Z'$-norm from \eqref{zprime} remains bounded:
\begin{equation}\label{eq:Z'bd}
    \norm{\gamma}_{Z'}+\norm{\nabla_x\gamma}_{Z'}\lesssim \eps.
\end{equation}

\textbf{Energy estimates: control of the spatial moments in \eqref{boot:in2}.} 
By \eqref{eq:gamma}, we have that
\begin{equation*}
    \partial_t (x_j \gamma) - \nabla_x \phi(x+tv,t) \cdot \{\nabla_v -t \nabla_x\}(x_j \gamma)=t \partial_{x_j} \phi(x+tv,t)\gamma,
\end{equation*}
and it follows that
\begin{equation*}
     \frac{d}{dt} \|x\gamma\|^2_{L^2_{x,v}}=t\|\gamma \nabla_x\phi(x+tv,t)\cdot(x\gamma)\|_{L^2_{x,v}}\leq t \|\nabla_x \phi\|_{L^\infty_x}\|\gamma\|_{L^2_{x,v}} \|x\gamma\|_{L^2_{x,v}}.
\end{equation*}
We conclude with Grönwall's inequality and the decay bound \eqref{boot:dec1} that 
\begin{equation*}
    \|x\gamma\|^2_{L^2_{x,v}}\le \|x\gamma_0\|^2_{L^2_{x,v}}e^{\int_0^t s \|\nabla_x \phi(s)\|_{L^\infty_x}\|\gamma(s)\|_{L^2_{x,v}} ds}\lesssim \|x\gamma_0\|^2_{L^2_{x,v}}.
\end{equation*}

\textbf{Energy estimates: control of spatial regularity in \eqref{boot:in2}.}
Let $\alpha\in\N_0^2$ be a multi-index with $1\leq|\alpha|\leq 3$. From \eqref{eq:gamma} and the Leibniz rule, we get
\begin{align*}
    \frac12 \frac{d}{dt} \|\partial_x^\alpha \gamma\|^2_{L^2_{x,v}}
    =  \sum_{\substack{\beta_1+\beta_2=\alpha \\
    |\beta_2|<|\alpha|}}c_{\beta_1, \beta_2} J_{\alpha}^{\beta_1,\beta_2},
\end{align*}
with
\begin{equation*}
J_{\alpha}^{\beta_1,\beta_2}:=
    \iint \nabla_x \partial^{\beta_1}_x \phi(x+tv,t)
    \cdot \{\nabla_v -t\nabla_v\}
    \partial_x^{\beta_2}\gamma
    \partial_x^\alpha \gamma dx dv,
\end{equation*}
where the case $|\beta_2|=|\alpha|$ in the summation is excluded by the divergence structure of the nonlinearity.
If $|\beta_1|\le 2$ we will use the direct estimate
\begin{equation*}
     |J_\alpha^{\beta_1,\beta_2}|
    \le \|\nabla_x \partial_x^{\beta_1}\phi\|_{L^\infty_x}
    \left( \|\nabla_v\partial_x^{\beta_2}\gamma\|_{L^2_{x,v}}+t \|\nabla_x\partial_x^{\beta_2}\gamma\|_{L^2_{x,v}}\right) \|\partial_x^{\alpha}\gamma\|_{L^2_{x,v}},
\end{equation*}
so that if $1\leq\abs{\alpha}\leq 2$ there holds that
\begin{equation}\label{sec3:h3last-1}
     \frac{d}{dt} \|\partial_x^\alpha \gamma\|_{L^2_{x,v}} 
    \lesssim 
    \sum_{\substack{\beta_1+\beta_2=\alpha \\
    |\beta_1|,|\beta_2|<|\alpha|}} \|\nabla_x \partial_x^{\beta_1}\phi\|_{L^\infty_x}
    \left(\|\gamma\|_{H^{|\beta_2|}_x H^1_v}+t \|\gamma\|_{H^{|\beta_2|+1}_x L^2_v}\right).
\end{equation}
On the other hand, when $|\beta_1|=3$ (and thus $\beta_2=0$) we note that
\begin{equation}
\label{sec3:J2}
\begin{aligned}
|J_\alpha^{\beta_1,0}| &\le 
\Big \| 
\nabla_x \partial_x^{\beta_1} \phi(x+tv,t)\cdot\{\nabla_v- t \nabla_x\}\gamma
\Big\|_{L^2_{x,v}}
\|\partial_x^{\alpha}\gamma\|_{L^2_{x,v}}\\
&\le
\|\nabla_x \partial_x^{\alpha}\phi\|_{L^2_x}
    \left( \|\nabla_v\gamma\|_{L^2_vL^\infty_x }+t 
    \|\nabla_x\gamma\|_{L^2_vL^\infty_x }\right)\|\partial_x^{\alpha}\gamma\|_{L^2_{x,v}},
\end{aligned}
\end{equation}
and thus when $\abs{\alpha}=3$ we have that
\begin{equation}
    \label{sec3:h3last-2}
\begin{aligned}
     \frac{d}{dt} \|\partial_x^\alpha \gamma\|_{L^2_{x,v}} 
    &\lesssim 
    \sum_{\substack{\beta_1+\beta_2=\alpha \\
    |\beta_1|,|\beta_2|<|\alpha|}} \|\nabla_x \partial_x^{\beta_1}\phi\|_{L^\infty_x}
    \left( \|\gamma\|_{H^{|\beta_2|}_x H^1_v}+t \|\gamma\|_{H^{|\beta_2|+1}_x L^2_v}\right)\\
    &\qquad +\|\nabla_x \partial_x^{\alpha}\phi\|_{L^2_x}
    \left( \|\gamma\|_{H^2_x H^1_v}+t 
    \|\gamma\|_{H^3_x L^2_v}\right).
\end{aligned}
\end{equation}
Since
\begin{equation*}
    \|\gamma\|_{H^1_x H^1_v}\le \|\gamma\|^{\frac12}_{H^2_x L^2_v} \|\gamma\|^{\frac12}_{L^2_x H^2_v}, \quad
    \|\gamma\|_{H^2_x H^1_v}\le \|\gamma\|_{H^3_x L^2_v}^{\frac23}
    \|\gamma\|_{L^2_x H^3_v}^{\frac13},
\end{equation*}
we obtain the claim by integrating \eqref{sec3:h3last-1} resp.\ \eqref{sec3:h3last-2} and using the decay estimates \eqref{boot:dec1}--\eqref{boot:dec3} for $\nabla_x\phi$ as well as the bootstrap assumptions \eqref{boot:hyp}.

\textbf{Energy estimates: control of velocity regularity in \eqref{boot:in2}--\eqref{boot:in4}.}
We are left to control the $L^2_x H^{|\alpha|}_v$, $|\alpha|\le 3$, and $H^2_xH^1_v$ norms of $\gamma$, which we do case by case for $|\alpha|\in \{1,2,3\}$.

\medskip
\emph{Case $|\alpha|=1$:} Here we will establish that $\|\gamma\|_{H^2_x H^1_v}$ remains uniformly bounded. We give the full details for the highest order derivatives, as the proof for lower order terms proceeds along similar but simpler lines.

For $|\beta|=2$, we compute from \eqref{eq:gamma} that
\begin{equation}
\label{eq:multlin-h11}
    \frac12 \frac{d}{dt}\|\partial^\beta_x \partial_{v_j} \gamma\|_{L^2_{x,v}}^2=
    \sum_{\substack{\alpha_1+\alpha_2=\beta \\ |\alpha_2| < 2}} c_{\alpha_1, \alpha_2} T_1^{\alpha_1, \alpha_2}
    + t \sum_{\substack{\alpha_1+\alpha_2=\beta}} c_{\alpha_1, \alpha_2} T_2^{\alpha_1,\alpha_2},
\end{equation}
where
\begin{align*}
    T_1^{\alpha_1, \alpha_2}&:=\iint
    \partial_x^{\alpha_1}\nabla_x \phi(x+tv,t)\cdot \{\nabla_v - t\nabla_x\} \partial_x^{\alpha_2} \partial_{v_j}\gamma  \;\partial_x^\beta \partial_{v_j}\gamma dxdv, \\
    T_2^{\alpha_1,\alpha_2}&:=\iint \partial_x^{\alpha_1}\partial_{x_j}\nabla_x\phi(x+tv,t)
    \cdot \{\nabla_v - t \nabla_x\} \partial_x^{\alpha_2} \gamma \;\partial_x^\beta \partial_{v_j}\gamma dxdv.
\end{align*}
Notice that in the first summation in \eqref{eq:multlin-h11} the case $|\alpha_2|=2$ is excluded by the divergence structure of the nonlinearity. In particular, there it suffices to bound
\begin{equation*}
|T_1^{\alpha_1, \alpha_2}|\le \|\partial_x^{\alpha_1}\nabla_x \phi\|_{L^\infty_x}\left( \|\gamma\|_{H^{|\alpha_2|}_x H^2_v}+ t \|\gamma\|_{H_x^{|\alpha_2|+1}H^1_v}\right) \|\partial_x^\beta \partial_{v_j} \gamma\|_{L^2_{x,v}}.
\end{equation*}
If $|\alpha_1|< 2$ the analogous bound for $T_2^{\alpha_1, \alpha_2}$ reads
    \begin{equation*}
        |T_2^{\alpha_1, \alpha_2}| \le
         \|\partial_x^{\alpha_1}\partial_{x_j}\nabla_x \phi\|_{L^\infty_x}\{ \|\gamma\|_{H^{|\alpha_2|}_x H^1_v}+ t \|\gamma\|_{H_x^{|\alpha_2|+1}L^2_v}\} \|\partial_x^\beta \partial_{v_j} \gamma\|_{L^2_{x,v}},
    \end{equation*}
while if $|\alpha_1|=2$ (and thus $\alpha_2=0$), we have as in \eqref{sec3:J2} that
\begin{equation*}
  |T_2^{\alpha_1, \alpha_2}|\le
    \|\partial_x^{\alpha_1} \partial_{x_j}\nabla_x\phi\|_{L^2_{x}} \left(\|\nabla_v \gamma\|_{L^2_vL_x^\infty }+ t \|\nabla_x\gamma\|_{L^2_vL^\infty_x }\right) \|\partial_x^\beta \partial_{v_j}\gamma\|_{L^2_{x,v}}.
\end{equation*}
Collecting these estimates, we get
\begin{align*}
    \frac{d}{dt}\|\partial^\beta_x \partial_{v_j} \gamma\|_{L^2_{x,v}}
    &\lesssim 
    \sum_{\substack{\alpha_1+\alpha_2=\beta \\ |\alpha_2| < 2}} 
    \|\partial_x^{\alpha_1}\nabla_x \phi\|_{L^\infty_x}\left( \|\gamma\|_{H^{|\alpha_2|}_x H^2_v}+ t \|\gamma\|_{H_x^{|\alpha_2|+1}H^1_v}\right)\\
    &\qquad+ t \sum_{\substack{\alpha_1+\alpha_2=\beta \\  
    |\alpha_1|<2}}\|\partial_x^{\alpha_1}\partial_{x_j}\nabla_x \phi\|_{L^\infty_x}\left( \|\gamma\|_{H^{|\alpha_2|}_x H^1_v}+ t \|\gamma\|_{H_x^{|\alpha_2|+1}L^2_v}\right) \\
    & \qquad +\|\partial_x^{\beta} \nabla_x\phi\|_{L^2_{x}} 
    \left(\|\gamma\|_{H^2_x H^1_v}+ t \|\gamma\|_{H^3_x L^2_v}\right).
\end{align*}
This yields the claimed bound for $\|\gamma\|_{H^2_x H^1_v}$ in \eqref{boot:in2} by integrating in time and using the assumption on the initial datum in \eqref{sec3:initialdat}, the bootstrap hypotheses in \eqref{boot:hyp}, the $L^\infty$ decay estimates in \eqref{boot:dec1},\eqref{boot:dec2} and the $L^2$ bound in \eqref{boot:dec3}.

\medskip
\emph{Case $|\alpha|= 2$:}
We compute that
\begin{equation}\label{eq:enest}
    \frac{1}{2} \frac{d}{d t} \| \partial^\alpha_v \gamma \|^2_{L^2_{x,v}}
    = \sum_{\substack{\beta_1 + \beta_2=\alpha \\ |\beta_2| < |\alpha|}}
    c_{\beta_1, \beta_2}\mathcal{T}_1^{\beta_1,\beta_2}
     - \sum_{\substack{\beta_1 + \beta_2=\alpha \\ |\beta_2| < |\alpha|}}c_{\beta_1, \beta_2}\mathcal{T}_2^{\beta_1, \beta_2},
\end{equation}
where
\begin{equation}
    \begin{aligned}
    \label{sec3:T12}
    \mathcal{T}_1^{\beta_1, \beta_2}&:=
    \iint_{\mathbb{R}^2_{x,v}} \nabla_x \partial^{\beta_1}_v\{ \phi(x+tv,t) \}\cdot \partial^{\beta_2}_v \nabla_v \gamma(x,v,t)  \partial_v^\alpha \gamma(x,v,t) d x d v, \\
    T_2^{\beta_1, \beta_2}&:=t\iint_{\mathbb{R}^2_{x,v}} 
    \nabla_x \partial^{\beta_1}_v \{\phi(x+tv,t) \}\cdot \partial^{\beta_2}_v \nabla_x\gamma(x,v,t) \partial_v^\alpha \gamma(x,v,t) d x d v.
 \end{aligned}
 \end{equation}
From the decay estimates \eqref{boot:dec1} and \eqref{boot:dec2} we directly see that the first term gives an acceptable contribution:
\begin{equation}
\label{sec3:recall}
\Big|\mathcal{T}_1^{\beta_1,\beta_2}\Big|
\le t^{|\beta_1|}\|\partial_x^{\beta_1} \nabla_x \phi \|_{L^\infty_x} \|\gamma\|^2_{L^2_x H^{|\alpha|}_v}\lesssim \varepsilon^2 t^{-2+3\delta} \|\gamma\|^2_{L^2_x H^{|\alpha|}_v}.
\end{equation}
Concerning the term $\mathcal{T}_2^{\beta_1,\beta_2}$ in \eqref{eq:enest}, we
use a multilinear expansion, replacing $\phi$ with the integral expression $G_2 * \rho_\gamma$ given in \eqref{sec2:greenconv}, which gives that
\begin{equation}
\label{sec3:exp1}
\mathcal{T}_2^{\beta_1,\beta_2}=t^{|\beta_1|+1} \int_{0}^\infty I_R^{\beta_1, \beta_2} e^{-R}\frac{dR}{4 \pi R},
\end{equation}
where
\begin{align*}
I_R^{\beta_1, \beta_2}&:= \iint_{\R^2} \gamma^2(y-tw,w,t) \partial_x^{\beta_1} \partial_{x_j}
\Big\{ \chi(R^{-\frac{1}{2}}(x-y))\Big\}\\
&\qquad\qquad \cdot\partial_{x_j}\partial_v^{\beta_2}\gamma(x-tv,v,t) \partial_v^\alpha \gamma(x-tv,v,t) dx dy dw dv.
\end{align*}
From the dispersive change of variables $u\mapsto y-tw=:a$ and $v\mapsto x-tv=:b$ we obtain that
\begin{equation*}
\begin{aligned}
    I_R^{\beta_1, \beta_2}
    =
    t^{-4} \iint_{\R^2} \gamma^2\Big(a, \frac{y-a}{t},t&\Big) \partial_x^{\beta_1} \partial_{x_j} 
    \Big\{ \chi(R^{-\frac12}(x-y)) \Big \} \\
    &\cdot\partial_{x_j} \partial_v^{\beta_2} \gamma
    \Big(b, \frac{x-b}{t}, t\Big) \partial_v^\alpha 
    \gamma\Big(b, \frac{x-b}{t},t\Big) dx dy da db.
\end{aligned}
\end{equation*}
Next we integrate by parts in the $y$ variable, using
\begin{equation*}
    \partial_x^{\beta_1} \partial_{x_j} 
    \Big\{ \chi(R^{-\frac12}(x-y)) \Big \}
    =
    (-1)^{|\beta_1|}\partial_y^{\beta_1} \partial_{x_j} 
    \Big\{ \chi(R^{-\frac12}(x-y)) \Big \},
\end{equation*}
and obtain that
\begin{equation}
\label{sec3:exp2}
\begin{aligned}
    I_R^{\beta_1, \beta_2}= t^{-4-|\beta_1|} \sum_{\substack{\theta_1+\theta_2=\beta_1, \\ |\theta_1|\le |\theta_2|}} c_{\theta_1, \theta_2} I_R^{\theta_1, \theta_2, \beta_2},
\end{aligned}
\end{equation}
where
\begin{equation}
\label{sec2:multlin}
\begin{aligned}
I_R^{\theta_1, \theta_2, \beta_2}:= \iint_{\R^2}
    \partial_v^{\theta_1}
    \gamma\Big(a, \frac{y-a}{t}, t\Big) &\partial_v^{\theta_2} \gamma\Big(a,\frac{y-a}{t},t\Big) \partial_{z_j} \Big \{ \chi(R^{-\frac12}z) \Big \} \\
    &\cdot \partial_{x_j} \partial_v^{\beta_2} 
    \gamma\Big(b, \frac{z+y-b}{t},t\Big) \partial_v^\alpha \gamma\Big(b, \frac{z+y-b}{t},t\Big) dz dy da db.
\end{aligned}
\end{equation}
To recap, from \eqref{sec3:exp1} and \eqref{sec3:exp2} we have that
\begin{equation}
\label{sec3:Tcal}
    \mathcal{T}_2^{\beta_1, \beta_2}=
t^{-3}\sum_{\substack{\theta_1+\theta_2=\beta_1, \\ |\theta_1|\le |\theta_2|}} c_{\theta_1, \theta_2} 
    \int_0^\infty I_{R}^{\theta_1, \theta_2, \beta_2} e^{-R} \frac{dR}{4 \pi R},
\end{equation}
and it suffices to bound appropriately the quantity $ I_{R}^{\theta_1, \theta_2, \beta_2}$ in \eqref{sec2:multlin}, where $\abs{\beta_2}<\abs{\alpha}=2$ and $\beta_1+\beta_2=\alpha$ by construction (see \eqref{eq:enest}). In what follows we will show that
\begin{equation}\label{eq:IR2_bd}
    \abs{I_R^{\theta_1, \theta_2, \beta_2}}\lesssim R^{\frac12}t^2 
    \varepsilon^2\|\gamma\|^2_{L^2_x H^2_v},
\end{equation}
which by \eqref{sec3:Tcal} implies
\begin{equation*}
\abs{\mathcal{T}_2^{\beta_1,\beta_2}}\lesssim t^{-1} \varepsilon^2 \|\gamma\|^2_{L^2_x H^2_v}.
\end{equation*}
Together with \eqref{eq:enest} and \eqref{sec3:recall}, by Gr\"onwall's inequality and the assumption on the initial datum in \eqref{sec3:initialdat} we then obtain the slowly growing bound $\|\gamma\|_{L^2_x H_v^2}$ stated in \eqref{boot:in3}.

To establish \eqref{eq:IR2_bd}, we start by considering the case $\beta_2=0$, where $\beta_1=\alpha$. If $\theta_1=0$, then $\theta_2=\beta_1=\alpha$ and we see that
\begin{equation}
\label{sec3:prot1}
\begin{aligned}
\Big|I_R^{0, \alpha, 0}\Big| &\le
    \left(\norm{
    \gamma\Big(a, \frac{y}{t},t\Big)}_{L_y^\infty L^2_a}+\Big\|
    \gamma\Big(a, \frac{y-a}{t},t\Big)-\gamma\Big(a, \frac{y}{t},t\Big)\Big\|_{L^2_aL_y^\infty }\right)\\
    &\qquad \cdot
    \Big\|\partial_v^{\alpha} \gamma\Big(a,\frac{y-a}{t},t\Big) 
    \Big\|_{L^2_{y,a}}
    \norm{\partial_{z_j} \Big \{ \chi(R^{-\frac12}z) \Big \}}_{L^1_z}\\
    &\qquad \cdot 
    \left(\Big\|\partial_{x_j}  
    \gamma\Big(b, \frac{z+y}{t},t\Big)\Big\|_{L^\infty_y L^2_b}+\Big\|\partial_{x_j}  
    \gamma\Big(b, \frac{z+y-b}{t},t\Big)-\partial_{x_j}  
    \gamma\Big(b, \frac{z+y}{t},t\Big)\Big\|_{L^2_bL^\infty_y }\right)\\
    &\qquad \cdot\Big\|\partial_v^\alpha \gamma\Big(b, \frac{z+y-b}{t},t\Big)\Big\|_{L^2_{y,b}}\\
  &\lesssim R^{\frac12}t^2 
  \|\gamma\|_{Z'}
    \|\nabla_x 
    \gamma\|_{Z'} \|\gamma\|^2_{L^2_x H^{2}_v}.
    \end{aligned}
    \end{equation}
By the boundedness of the $Z$-norms in \eqref{boot:in1}, we get
    \begin{equation*}       \abs{I_R^{0,\alpha,0}}\lesssim R^{\frac12}t^2 \varepsilon^2 \| \gamma\|^2_{L^2_xH^2_v}.
\end{equation*}
If $\beta_2=0$ and $\abs{\theta_1}=\abs{\theta_2}=1$, then by Hölder's inequality (in the $y$ variable) we have
\begin{equation}
\label{sec3:prot2}
\begin{aligned}
    \abs{I_R^{\theta_1, \theta_2, 0}}
    &\le 
    \Big\|
    \partial_{v}^{\theta_1}\gamma\Big(a, \frac{y-a}{t},t\Big) \Big\|_{L^2_aL^4_y }
    \Big\|\partial_v^{\theta_2} \gamma\Big(a,\frac{y-a}{t},t\Big) 
    \Big\|_{ L^2_aL^4_y}
    \norm{\partial_{z_j} \Big \{ \chi(R^{-\frac12}z) \Big \}}_{L^1_z}\\
    &\qquad \cdot 
    \left(\Big\|\partial_{x_j}  
    \gamma\Big(b, \frac{z+y}{t},t\Big)\Big\|_{L^\infty_y L^2_b}+\Big\|\partial_{x_j}  
    \gamma\Big(b, \frac{z+y-b}{t},t\Big)-\gamma\Big(b, \frac{z+y}{t},t\Big)\Big\|_{L^2_bL^\infty_y }\right)\\
    &\qquad\cdot\Big\|\partial_v^\alpha \gamma\Big(b, \frac{z+y-b}{t},t\Big)\Big\|_{L^2_{b,y}}\\
    &\lesssim R^{\frac12}t^2\norm{\nabla_v\gamma}_{L^2_xL^4_v}^2
\norm{\nabla_x\gamma}_{Z'}\norm{\gamma}_{L^2_{x} H^2_v}\\
    &\lesssim R^{\frac12}
    t^2\norm{\gamma}_{L^2_{x}H^1_v}\norm{\nabla_x\gamma}_{Z'}\norm{\gamma}^2_{L^2_{x} H^2_v},
\end{aligned}
\end{equation}
where in the last step we have used the Gagliardo-Nirenberg inequality
\begin{equation}
\label{sec3:GNS}
    \|\nabla_v\gamma\|_{L^2_xL^4_v }\lesssim \|\gamma\|_{L^2_{x}H^1_v}^{\frac12}
    \|\gamma\|_{L^2_x H^2_v}^{\frac12}.
\end{equation}
By the already established boundedness of $\|\nabla_x \gamma\|_{Z'}$ in \eqref{boot:in1} resp.\ \eqref{eq:Z'bd} 
and of $\|\gamma\|_{L^2_x H^1_v}$ in \eqref{boot:in2}, we get
\begin{equation*}
    \abs{I_R^{\theta_1, \theta_2, 0}}
    \lesssim
    R^{\frac12}t^2 \varepsilon^2 \| \gamma\|^2_{L^2_xH^2_v}.
\end{equation*}

It remains to study the case $\abs{\beta_2}=1$. Here we have $\abs{\beta_1}=1$ and $\theta_2=\beta_1$, $\theta_1=0$. In this case, we integrate by parts once more:
\begin{equation}
\label{mult:intpart1}
    I_R^{0,\beta_1, \beta_2}= t^{-1}  \sum_{(\theta_3,\theta_4)\in\{(e_j,\beta_1),(0,\beta_1+e_j)\}} c_{\theta_3, \theta_4}\Im_R^{\theta_3, \theta_4},
\end{equation}
where
\begin{equation}
\label{mult:intpart}
    \begin{aligned}
        \Im_R^{\theta_3, \theta_4}:= \iint_{\R^2} 
        \partial_v^{\theta_3}
        \gamma\Big(a, \frac{y-a}{t},t\Big)&\partial_v^{\theta_4} \gamma\Big(a,\frac{y-a}{t},t\Big)
        \chi(R^{-\frac12}(x-y))\\
    &\cdot \partial_{x_j} \partial_v^{\beta_2} \gamma\Big(b, \frac{x-b}{t},t\Big) \partial_v^\alpha \gamma\Big(b, \frac{x-b}{t},t\Big) dx dy da db.
    \end{aligned}
\end{equation}
When $\theta_3=0$, $|\theta_4|=|\alpha|$, we have that
\begin{equation*}
\begin{aligned}
    \abs{\Im_R^{0,\theta_4}}
    &\le
    \Big\|
    \gamma\Big(a, \frac{y-a}{t},t\Big) \Big\|_{L^\infty _y L^2_a}
    \Big\|\partial_v^{\theta_4} \gamma\Big(a,\frac{y-a}{t},t\Big) 
    \Big\|_{L^2_{y,a}}
    \norm{\chi(R^{-\frac12}(x-y))}_{L^2_y}\\
    &\qquad \cdot 
    \Big\|\partial_{x_j}  \partial_v^{\beta_2}
    \gamma\Big(b, \frac{x-b}{t},t\Big)\Big\|_{L^2_{x,b}} \Big\|\partial_v^\alpha \gamma\Big(b, \frac{x-b}{t},t\Big)\Big\|_{L^2_{x,b}} \\
    &\lesssim R^{\frac12} t^3 
    \|\gamma\|_{Z'}\|\gamma\|_{H^1_xH^1_v}\|\gamma\|_{L^2_x H^2_v}^2.
\end{aligned}
\end{equation*}
By the already proved boundedness of $\|\gamma\|_{Z'}$
in \eqref{boot:in1} resp.\ \eqref{eq:Z'bd} and $\|\gamma\|_{H^1_xH^1_v}$ in \eqref{boot:in2}, we conclude that
\begin{equation}
    \label{mult:in3}
    \abs{\Im_R^{0,\theta_4}}\lesssim R^{\frac12} t^3 \varepsilon^2 \|\gamma\|^2_{L^2_x H^2_v}.
\end{equation}
Similarly, if $|\theta_3|=|\theta_4|=1$ we obtain as above 
\begin{equation*}\begin{aligned}
 \abs{\Im_R^{\theta_3, \theta_4}}
    &\le
    \Big\|\nabla_v \gamma\Big(a, \frac{y-a}{t},t\Big)\Big\|^2_{L^2_{a}L^4_y }
    \left\|\chi\left(R^{-\frac12}(x-y)\right)\right\|_{L^2_y}\\
    &\qquad \cdot
\Big\|\partial_{x_j}\partial_v^{\beta_2} \gamma\Big(b, \frac{x-b}{t},t\Big)
\Big\|_{L^2_{b,x}}
    \Big\|\partial_v^{\alpha} \gamma\Big(b, \frac{x-b}{t},t\Big)\Big\|_{L^2_{b,x}}\\
    &\lesssim 
    R^{\frac12}t^3 \|\nabla_v\gamma\|^2_{L^2_{x}L^4_v} \norm{\gamma}_{H^1_xH^1_v}\norm{\gamma}_{L^2_x H^2_v}\\
    &\lesssim R^{\frac12} t^3 \norm{\gamma}_{L^2_x H^1_v}\norm{\gamma}_{H^1_x H^1_v}\|\gamma\|^2_{L^2_x H^2_v},
\end{aligned}
\end{equation*}
where we used again \eqref{sec3:GNS}.
By the already proved boundedness of $\|\gamma\|_{H^1_xH^1_v}$ in \eqref{boot:in2}, it follows that
\begin{equation}
    \label{mult:in4}
     \abs{\Im_R^{\theta_3, \theta_4}}\lesssim
     R^{\frac12} t^3 \varepsilon^2\|\gamma\|^2_{L^2_x H^2_v}.
\end{equation}

Collecting \eqref{mult:in3} and \eqref{mult:in4} inside \eqref{mult:intpart1}, we obtain
\begin{equation*}
    \abs{I_R^{0, \beta_1, \beta_2}}\lesssim R^{\frac12}t^2 
    \varepsilon^2\|\gamma\|^2_{L^2_x H^2_v},
\end{equation*}
finishing the proof of \eqref{eq:IR2_bd}.

\medskip
\emph{Case $|\alpha|=3$:}
We use the decomposition \eqref{eq:enest} and \eqref{sec3:T12}, and study first $\mathcal{T}_1^{\beta_1,\beta_2}$. If $|\beta_1|, |\beta_2|\le 2$,  we can proceed as in \eqref{sec3:recall}, getting
\begin{equation}
\label{est3}
\Big|\mathcal{T}_1^{\beta_1,\beta_2}\Big|\lesssim \varepsilon^2 t^{-2+3\delta} \|\gamma\|^2_{L^2_x H^{3}_v}.
\end{equation}
Otherwise, if all the derivatives fall on $\nabla_x \phi$ (and thus $\beta_1=\alpha$), we change variables
\begin{equation*}
\begin{aligned}
\mathcal{T}_1^{\alpha,0}
&=t^{|\alpha|-2} \iint_{\mathbb{R}^2_{x,v}} 
\nabla_x \partial^{\alpha}_x \phi(x+tv,t) \cdot  \nabla_v \gamma\Big(a,\frac{x-a}{t},t\Big) \partial_v^\alpha \gamma\Big(a,\frac{x-a}{t},t\Big) 
d x d a 
\end{aligned}
\end{equation*}
and get
\begin{equation*}
\begin{aligned}
\abs{\mathcal{T}_1^{\alpha,0}}&\le t^{|\alpha|-2} \|\nabla_x \partial_x^\alpha \phi \|_{L^2_x} \Big\|\nabla_v \gamma\Big(a, \frac{x-a}{t},t\Big)\Big\|_{L^2_a L^\infty_x}
\Big\|\partial_v^\alpha\gamma\Big(a, \frac{x-a}{t},t\Big)\Big\|_{L^2_{x,a}}\\
&\le t^{|\alpha|-1} \|\nabla_x \partial_x^\alpha \phi\|_{L^2_x} \|\nabla_v \gamma \|_{L^2_x L^\infty_v} 
\|\gamma\|_{L^2_x H^3_v} 
\lesssim 
t^{|\alpha|-1} \|\nabla_x \partial_x^\alpha \phi\|_{L^2_x} 
\|\gamma\|^2_{L^2_x H^3_v}\\
&\lesssim \varepsilon^2 \ip{t}^{-\frac32+2\delta} \|\gamma\|^2_{L^2_x H^3_v},
\end{aligned}
\end{equation*}
where in the last inequality we used the Sobolev embedding $\|\nabla_v\gamma\|_{L^2_x L^\infty_v} \lesssim \|\gamma\|_{L^2_x H^3_v}$ and the decay estimate \eqref{boot:dec3}.

We now pass to study the term $\mathcal{T}_2^{\beta_1, \beta_2}$ in \eqref{sec3:T12}. We perform the multilinear expansion as in \eqref{sec3:Tcal} and we study the term $I_R^{\theta_1,\theta_2, \beta_2}$ in \eqref{sec2:multlin}. 

If $\beta_2=\theta_1=0$ and $\theta_2=\alpha$, then we proceed as in \eqref{sec3:prot1} and we get
\begin{equation}
\label{est5}
    \abs{I_R^{0, \alpha, 0}} \lesssim R^{\frac12} t^2 \|\gamma\|_{Z'} \|\nabla_x \gamma \|_{Z'} \|\partial_v^\alpha \gamma\|^2_{L^2_{x,v}}\lesssim R^{\frac12}t^2 \varepsilon^2 \| \gamma\|^2_{L^2_xH^3_v}.
\end{equation}
If $\beta_2=0$, $|\theta_1|=1$ and $|\theta_2|=2$, 
arguing as in \eqref{sec3:prot2}, we have
\begin{equation*}
\begin{aligned}
    \abs{I_R^{\theta_1, \theta_2, 0}}
    &\lesssim R^{\frac12} t^2 \|\partial_v^{\theta_1} \gamma\|_{L^2_x L^4_v}\|\partial^{\theta_2}_v \gamma \|_{L^2_x L^4_v}\|\nabla_x\gamma\|_{Z'} \|\partial_v^\alpha \gamma\|_{L^2_{x,v}}\\ 
    &\lesssim R^{\frac12} t^2 \|\nabla_x \gamma\|_{Z'}
    \|\gamma\|^{\frac12}_{L^2_{x}H^1_v}
    \|\gamma\|_{L^2_x H^2_v}\|\gamma\|^{\frac32}_{L^2_x H_v^{3}}
\end{aligned}\end{equation*}
where in the last inequality we have used \eqref{sec3:GNS} for
$\|\partial_v^{\theta_1} \gamma\|_{L^2_x L^4_v}$ and $\|\partial^{\theta_2}_v \gamma \|_{L^2_x L^4_v}$.
By the already established bounds \eqref{boot:in2} and \eqref{boot:in3} we conclude
\begin{align}
\label{est6}
 \abs{I_R^{\theta_1, \theta_2, 0}}
 \lesssim R^{\frac12} t^2 \varepsilon^2
 \Big[ \varepsilon^{\frac12}
 \ip{t}^{C\varepsilon^2}\|\gamma\|^{\frac32}_{L^2_x H^3_v}
 \Big]
 \lesssim R^{\frac12} t^2 
 \varepsilon^2\left[\varepsilon^2 \ip{t}^{6 C\varepsilon^2}+ \frac{\|\gamma\|^2_{L^2_x H^3_v}}{\ip{t}^{\frac23 C\varepsilon^2}}\right],
 \end{align}
where in the last step we applied Young's inequality.

If $|\beta_2|=1$, we consider first the subcase $\theta_1=0$ and thus $|\theta_2|=|\beta_1|=2$. We have
 \begin{equation*}
\begin{aligned}
    \abs{I_R^{0, \theta_2, \beta_2}}
    &\lesssim 
    \norm{\gamma\left(a, \frac{y+z-a}{t},t\right)}_{L^\infty_y L^2_a} \Big\|\partial^{\theta_2}_v \gamma\Big(a, \frac{y+z-a}{t},t\Big)\Big\|_{L^4_y L^2_a} \Big\|\nabla_z\Big\{ \chi(R^{-\frac12}z)\Big\}\Big\|_{L_z^{1}}\\
    &\qquad \cdot
    \Big\|\partial_{x_j} \partial_v^{\beta_2}\gamma\Big(b, \frac{y-b}{t},t\Big)\Big\|_{L^4_y L^2_{b}} 
    \Big\| \partial_v^\alpha \gamma\Big(b, \frac{y-b}{t},t\Big)\Big\|_{L^2_{y,b}}\\
    & \lesssim  R^{\frac12}t^{2}  \|\gamma\|_{Z'}
    \|\partial_v^{\theta_2}\gamma\|_{L^2_{x,v}}^{\frac12}
    \|\gamma\|_{L^2_x H^3_v}^{\frac12}
    \|\gamma\|^{\frac12}_{H^1_{x}H^1_v}
    \|\gamma\|^{\frac12}_{H^1_x H^2_v}
    \|\gamma\|_{L^2_x H^{3}_v},
\end{aligned}
\end{equation*}
where in the last inequality we used again \eqref{sec3:GNS}.
With
$\|\gamma\|_{H^1_x H^2_v}\lesssim \|\gamma\|_{H^2_x H^1_v}^{\frac12}\|\gamma\|_{L^2_x H^3_v}^{\frac12}$ and \eqref{boot:in2}, \eqref{boot:in3} we conclude that
\begin{equation}
\label{est7}
 \abs{I_R^{0, \theta_2, \beta_2}}
 \lesssim  R^{\frac12}t^{2}  \varepsilon^{2}
 \Big[\varepsilon^{\frac14}\ip{t}^{\frac{C\varepsilon^2}{2}} \|\gamma\|^{\frac74}_{L^2_x H^3_v} \Big]\lesssim R^{\frac12} t^2 \varepsilon^2 \left[
    \varepsilon^2 \ip{t}^{8\varepsilon^2} +\frac{\|\gamma\|^2_{L^2_x H^3_v}}{\ip{t}^{\frac47 \varepsilon^2}}
    \right ].
    \end{equation}

On the other hand, if $|\beta_2|=1$ and $|\theta_1|=|\theta_2|=1$, we integrate by parts as in \eqref{mult:intpart1}--\eqref{mult:intpart}. In this case we have to study $\Im_R^{\theta_3, \theta_4}$ with $|\theta_3|=1, |\theta_4|=2$.
  We get
\begin{equation*}
    \begin{aligned}
    \abs{\Im_R^{\theta_3, \theta_4}}
    &\lesssim \Big\|\partial_v^{\theta_3}\gamma\Big(a, \frac{x-a}{t},t\Big)\Big\|_{L^2_a L^4_x} \Big\|\partial_v^{\theta_4}\gamma\Big(a, \frac{x-a}{t},t\Big)\Big\|_{L^2_a L^4_x}
    \norm{\chi\left(R^{-\frac12}(x-y)\right)}_{L_x^{2}}\\
    &\qquad\cdot \Big\|\partial_{x_j} \partial_v^{\beta_2} \gamma\Big(b, \frac{y-b}{t},t\Big)\Big\|_{L^2_{y,b}} 
    \Big\| \partial_v^\alpha \gamma\Big(b, \frac{y-b}{t},t\Big)\Big\|_{L^2_{y,b}}\\
    & \lesssim  R^{\frac12}t^{3} \|\gamma\|^{\frac12}_{L^2_xH^1_v} \|\gamma\|_{L^2_x H^2_v}
    \|\gamma\|_{H^1_x H^1_v}
    \|\gamma\|^{\frac32}_{L^2_x H^3_v},
\end{aligned}
\end{equation*}
and with \eqref{boot:in2}, \eqref{boot:in3} it follows that
\begin{equation}
\label{sec8}
    \abs{\Im_R^{\theta_3, \theta_4}}
    \lesssim
    R^{\frac12} t^3
\varepsilon^{2}
\Big[
\varepsilon^{\frac12}
\ip{t}^{C\varepsilon^2} \|\gamma\|^{\frac32}_{L^2_x H^3_v}\Big]
    \lesssim
    R^{\frac12} t^3 \varepsilon^2 \left[\varepsilon^2 \ip{t}^{6C \varepsilon^2}+ \frac{\|\gamma\|^2_{L^2_x H^3_v}}{\ip{t}^{\frac23C \varepsilon^2}}\right].
\end{equation}
 
Lastly, if $|\beta_2|=2$ then $\theta_1=0$ and $|\theta_2|=1$, and we use again \eqref{mult:intpart1}--\eqref{mult:intpart}. There, either $|\theta_3|=|\theta_4|=1$
or $|\theta_3|=0, |\theta_4|=2$. If $|\theta_3|=|\theta_4|=1$, we obtain as above that
\begin{equation*}
    \begin{aligned}
    \abs{\Im_R^{\theta_3, \theta_4}}
    &\lesssim 
    \Big\|\nabla_v\gamma\Big(a, \frac{x-a}{t},t\Big)\Big\|^2_{L^2_a L^4_x} 
    \norm{\chi\left(R^{-\frac12}(x-y)\right)}_{L_x^{2}}\\
    &\qquad \cdot \Big\|\partial_{x_j} \partial^{\beta_2}_v \gamma\Big(b, \frac{y-b}{t},t\Big)\Big\|_{L^2_{y,b}} 
    \Big\| \partial_v^\alpha \gamma\Big(b, \frac{y-b}{t},t\Big)\Big\|_{L^2_{y,b}}\\
    & \lesssim  R^{\frac12}t^{3} \|\gamma\|_{L^2_xH^1_v} \|\gamma\|_{L^2_x H^2_v}
    \|\gamma\|_{H^1_x H^2_v}
    \|\gamma\|_{L^2_x H^3_v}\\
    &\lesssim R^{\frac12} t^3 \|\gamma\|_{L^2_xH^1_v} \|\gamma\|_{L^2_x H^2_v}
    \|\gamma\|_{H^2_x H^1_v}^{\frac12}
    \|\gamma\|^{\frac32}_{L^2_x H^3_v},
\end{aligned}
\end{equation*}
and hence
\begin{equation}
\label{est9}
  \abs{\Im_R^{\theta_3, \theta_4}}\lesssim R^{\frac12}t^{3} \varepsilon^{\frac52}\ip{t}^{C\varepsilon^2} \|\gamma\|^{\frac32}_{L^2_x H^3_v}
    \lesssim R^{\frac12}t^{3}\varepsilon^2\left[\varepsilon^2 \ip{t}^{6C\varepsilon^2}+ \frac{\|\gamma\|^2_{L^2_x H^3_v}}{\ip{t}^{\frac23 C\varepsilon^2}}\right].
    \end{equation}
If $|\theta_3|=0, |\theta_4|=2$, analogously we obtain
    \begin{equation*}
    \begin{aligned}
    \abs{\Im_R^{0, \theta_4}}
    &\lesssim \Big\|\gamma\Big(a, \frac{x-a}{t},t\Big) \Big\|_{L^\infty_x L^2_a}
\Big\|
\partial_v^{\theta_2}\gamma\Big(a, \frac{x-a}{t},t\Big)\Big\|_{L^2_{x,a}} 
    \norm{\chi\left(R^{-\frac12}(x-y)\right)}_{L_x^{2}}\\
    &\quad \times \Big\|\partial_{x_j} \partial^{\beta_2}_v \gamma\Big(b, \frac{y-b}{t},t\Big)\Big\|_{L^2_{y,b}} 
    \Big\| \partial_v^\alpha \gamma\Big(b, \frac{y-b}{t},t\Big)\Big\|_{L^2_{y,b}}\\
    & \lesssim  R^{\frac12}t^{3} \|\gamma\|_{Z'}
    \|\gamma\|_{L^2_x H^2_v}
    \|\gamma\|_{H^1_x H^2_v}
    \|\gamma\|_{L^2_x H^3_v},
\end{aligned}
\end{equation*}
and thus
\begin{equation}
\label{est10}
 \abs{\Im_R^{0, \theta_4}}\lesssim R^{\frac12}t^{3} \varepsilon^{\frac52}\ip{t}^{C\varepsilon^2} \|\gamma\|^{\frac32}_{L^2_x H^3_v}
    \lesssim R^{\frac12}t^{3}\varepsilon^2\left[\varepsilon^2 \ip{t}^{6C\varepsilon^2}+ \frac{\|\gamma\|^2_{L^2_x H^3_v}}{\ip{t}^{\frac23 C\varepsilon^2}}\right].
\end{equation}
    
Collecting the estimates \eqref{est5}--\eqref{est10}, we conclude that
\begin{equation*}
\abs{\mathcal{T}_2^{\beta_1, \beta_2}} 
\lesssim t^{-1}\varepsilon^2\left[\varepsilon^2 \ip{t}^{8C\varepsilon^2}+ \|\gamma\|^2_{L^2_x H^{3}_v}\right].
\end{equation*}
Together with \eqref{est3} and the assumptions on the initial data in \eqref{sec3:initialdat}, by Grönwall's inequality we obtain the claimed bound \eqref{boot:in4}.
\end{proof}

\begin{remark}[Propagation of regularity and sharp decay]\label{rem:sharpened2d} 
Reasoning as in Proposition \ref{sec3:boot}, one can propagate boundedness of higher regularity in $x,v$, and of higher order spatial moments as well. More precisely, for an arbitrary $n\in\N$ one could propagate $H^n_{x,v}$ estimates on $\gamma$, with $H^n_xL^2_v$ and $L^2_xH_v^{n-2}$ norms uniformly bounded and $L^2_xH^{n}_v$ bounds slowly growing. In particular, provided that at least $\norm{\gamma}_{L^2_x H^3_v}$ remains bounded, we would obtain the sharp decay estimate for $\nabla_x\phi$, since as in \eqref{eq:nablaphiibp} we have that
\begin{equation*}
    \norm{\nabla_x \phi}_{L^\infty_x}\lesssim 
    \ip{t}^{-3}\norm{\gamma}_{L^2_x L^\infty_v}
    \norm{\nabla_v\gamma}_{L^2_x L^\infty_v }\lesssim
    \ip{t}^{-3} \norm{\gamma}^2_{L^2_x H^3_v}.
\end{equation*}
\end{remark}

\subsection{Free scattering}\label{sec:2dfree_scat}
We can now conclude the proof of Theorem \ref{intro:thm1} in the case $d=2$. We state it here for the unknown $\gamma$, and note that the result for $\mu(x,v,t)=\gamma(x-tv,v,t)$ with $\gamma(x,v,0)=\mu_0(x,v)$ follows directly.
\begin{theorem}[Free scattering in $d=2$]
\label{sec3:main}
    If $d=2$, there exists $\eps_0>0$, such that if 
    \begin{equation*}
    \| x  \mu_0\|_{L^2_{x,v}}+\norm{\mu_0}_Z+\norm{\nabla_x\mu_0}_Z+\|\mu_0\|_{H^3_{x,v}} \le \eps\leq \eps_0,
    \end{equation*} 
   then there exists a unique, global solution $\gamma\in C_t(\R,H^3_{x,v}(\R^2\times\R^2))$ of \eqref{eq:gamma} which satisfies moreover
    \begin{equation}
    \label{sec3:est2}
       \|x \gamma(t)\|_{L^2_{x,v}} + \norm{\gamma(t)}_Z+\norm{\nabla_x\gamma(t)}_Z +\|\gamma(t)\|_{L^2_v H^3_x} \lesssim \eps, \quad \|\gamma(t)\|_{L^2_x H^3_v}\lesssim \eps \left \langle t \right \rangle^{C\eps^2},
    \end{equation}
    for a universal constant $C>0$. Moreover, there exists an asymptotic distribution function $\gamma_\infty \in H^1_{x,v} \cap Z$ such that
    \begin{equation}
    \label{sec3:scatt}
        \lim_{t \to \infty} \| \gamma(x,v,t) - \gamma_\infty(x,v)\|_{H^1_{x,v} \cap Z}= 0.
    \end{equation}
\end{theorem}

\begin{proof}
   The existence of a unique global solution $\gamma$
   of \eqref{eq:gamma} as claimed follows by classical local-wellposedness and the by propagation of the bounds on moments and derivatives
   in Proposition \ref{sec3:boot}.
    It remains to show the asymptotic convergence property \eqref{sec3:scatt}: From \eqref{eq:gamma} and Propositions \ref{prop:decay2} and \ref{sec3:boot}, we have by \eqref{sec3:est2} that for some $0<\delta\ll 1$ sufficiently small
    \begin{align*}
        \|\partial_t \gamma(t) \|_{Z} &\le \|\nabla_x \phi(t)\|_{L^\infty_x} (\|\nabla_v \gamma(t)\|_{Z} + t \|\nabla_x \gamma(t)\|_{Z})\\
        &\lesssim \eps^3 \ip{t}^{-2+\delta+C\eps^2},\\
        \|\partial_t \partial_{x_j}\gamma(t)\|_{L^2_{x,v}}&\le \norm{\nabla_x\phi(t)}_{L^\infty}(\|\partial_{x_j}\nabla_v \gamma(t)\|_{L^2_{x,v}} + t \|\partial_{x_j}\nabla_x \gamma(t)\|_{L^2_{x,v}})\\
        &\qquad +\|\partial_{x_j} \nabla_x \phi(t)\|_{L^\infty_x} (\|\nabla_v \gamma(t)\|_{L^2_{x,v}} + t \|\nabla_x \gamma(t)\|_{L^2_{x,v}})\\
        &\lesssim \eps^3 \ip{t}^{-2+\delta+C\eps^2},\\
        \|\partial_t \partial_{v_k}\gamma(t)\|_{L^2_{x,v}}&\le \norm{\nabla_x\phi(t)}_{L^\infty}(\|\partial_{v_k}\nabla_v \gamma(t)\|_{L^2_{x,v}} + t \|\partial_{v_k}\nabla_x \gamma(t)\|_{L^2_{x,v}})\\
        &\qquad +t\|\partial_{x_k} \nabla_x \phi(t)\|_{L^\infty_x} (\|\nabla_v \gamma(t)\|_{L^2_{x,v}} + t \|\nabla_x \gamma(t)\|_{L^2_{x,v}})\\
        &\lesssim \eps^3 \ip{t}^{-2+\delta+C\eps^2},
    \end{align*}
    so that $\partial_t \gamma$ is time-integrable in $H^1_{x,v}\cap Z$ and the existence of $\gamma_\infty\in H^1_{x,v} \cap Z$ satisfying \eqref{sec3:scatt} follows.
\end{proof}

\section{Long-time stability in $d=1$}
\label{sec:1d}
In this section we work in analytic functional setting. For $\lambda>0$ and $f:\R_x\times\R_v\to\R$ we consider the spatial and velocity regularity norms 
 \begin{equation*}
     \mathcal{E}_{v}^{\lambda}[f]:= \sum_{n=0}^\infty \frac{\lambda^n}{n!}\norm{\partial_v^n f}_{L^2_{x,v}}, \qquad
     \mathcal{E}_{x}^{\lambda}[f]:= \sum_{n=0}^\infty \frac{\lambda^n}{n!}\norm{\partial_x^n f}_{L^2_{x,v}}.
 \end{equation*}
The rest of this section is devoted to the proof of Theorem \ref{intro:thm1d}, which we state here again for convenience:
\begin{theorem}[Long-time stability in $d=1$]\label{thm:1d}
    There exists a universal constant $C>0$ such that the following holds. For given $R>0$, there exists $\eps_0\in (0,C^{-\frac12}R^{\frac12})$ such that if $\gamma_0$ satisfies
    \begin{equation}\label{eq:1d_assumptions}
        \sum_{a\in\{0,1\}}\left(\mathcal{E}_{x}^{R}[\partial_x^a\gamma_0]
        + \mathcal{E}_{v}^{R}[\partial_v^a\gamma_0]\right) \leq \eps \le \varepsilon_0,
    \end{equation}
    then there exist $T\gtrsim R^{2}\eps^{-4}$ and a unique solution $\gamma\in C_t([0,T],\cap_{k=1}^\infty H^k_{x,v}(\R\times\R))$ to the screened Vlasov-Poisson equation \eqref{eq:gamma} with $\gamma(x,v,0)=\gamma_0(x,v)$, which moreover satisfies
    \begin{equation*}
    \sum_{a\in\{0,1\}}\left(\mathcal{E}^{\lambda_t}_v[\partial_v^a\gamma(t)]+ \mathcal{E}_x^{\lambda_t}[\partial_x^a\gamma(t)]\right)\lesssim \eps,\quad \lambda_t=\lambda_t(R,\eps) := R-C\eps^2\ip{t}^{\frac12},\quad 0\leq t\leq T.
    \end{equation*}
\end{theorem}
We highlight that the time of existence improves as the analyticity-parameter $R$ resp.\ the size $\eps$ of the initial data increases resp.\ decreases. 

The proof of this result proceeds by first demonstrating some linear decay estimates in Section \ref{sec:1d_lin}, whereupon a bootstrap in Section \ref{sec:1d_nonlin} can be established. 

\subsection{Linear estimates}\label{sec:1d_lin}
Towards linear decay estimates, for $\lambda>0$ we define the analytic norms
 \begin{equation*}
     \mathcal{N}^\lambda_{v}[\phi_h(t)]:=\sum_{n=0}^\infty \frac{\lambda^n}{n!} \ip{t}^n\norm{\partial_x^n\phi_h(t)}_{L^\infty_x(\R)},\qquad
     \mathcal{M}^\lambda_{v}[h(t)]:=\sum_{n=0}^\infty \frac{\lambda^n}{n!} \norm{\partial_v^n h(t)}_{L^2_x L^\infty_v},
 \end{equation*}
where $\phi_h$ is defined as in \eqref{sec2:greenconv}. 
We highlight that the order of the Lebesgue spaces in $(x,v)$ in the definition of the $\mathcal{M}$-norm is reversed with respect to the $Z$-norm defined in \eqref{znorm}. This allows for simpler linear estimates, but we underline that in the proof of Theorem \ref{thm:1d} it will only be used via embeddings in the $L^2$-based energy norms: we observe that
 \begin{align*}
\mathcal{M}^\lambda_v[h]=\sum_{n=0}^\infty \frac{\lambda^n}{n!}\norm{\partial^n_vh}_{L^2_x L^\infty_v} \lesssim \sum_{n=0}^\infty \frac{\lambda^n}{n!}\norm{\partial^n_vh}^{\frac12}_{L^2_{x,v}}\norm{\partial^{n+1}_v h}^{\frac12}_{L^2_{x,v}}\lesssim\mathcal{E}^\lambda_v[h]^{\frac12}\mathcal{E}^\lambda_v[\partial_v h]^{\frac12}.
 \end{align*}
For future reference we also record some interpolation estimates for mixed derivatives, namely
\begin{equation}
    \label{sec7:interp}
\begin{aligned}
        \mathcal{E}^\lambda_x[\partial_v h]=\sum_{n=0}^\infty \frac{\lambda^{n}}{n!}\norm{\partial_v \partial_x^n h}_{L^2_{x,v}}
        &\le
        \sum_{n=0}^\infty \frac{\lambda^{n}}{n! }\norm{\partial_x^{n+1}h}_{L^2_{x,v}}^{\frac{n}{n+1}} \norm{\partial_v^{n+1} h}^{\frac{1}{n+1}}_{L^2_{x,v}}\\
        &\le\sum_{n=0}^\infty \frac{\lambda^{n}}{n!}
        \left(
        \frac{n}{n+1}\norm{\partial_x^{n+1}h}_{L^2_{x,v}}
        +
        \frac{1}{n+1}\norm{\partial_v^{n+1} h}_{L^2_{x,v}}
        \right)\\
        &\le  \mathcal{E}^\lambda_x[\partial_x h]+ \mathcal{E}^\lambda_v[\partial_v h],
    \end{aligned}
    \end{equation}
and analogously    
    \begin{align}
    \label{sec7:interp1}
        \mathcal{E}^\lambda_x[\partial_v^2 h]\le\mathcal{E}^\lambda_x[\partial^2_x h]+\mathcal{E}^\lambda_v[\partial^2_v h].
    \end{align}    

Our decay estimates read as follows.
\begin{lemma}\label{lem:1d_decay}
For $\lambda>0$, there holds that
        \begin{align}
        \mathcal{N}^\lambda_v[\ip{t}\partial_x\phi_h(t)]&\lesssim \ip{t}^{-1/2}\left(\mathcal{M}^\lambda_v[h]\mathcal{E}^\lambda_v[\partial_vh]+\mathcal{E}^\lambda_x[\partial_xh]\mathcal{E}^\lambda_x[h]\right), \label{sec7:bound1N}\\
        \mathcal{N}^\lambda_v[\ip{t}^2\partial_x^2\phi_h(t)]&\lesssim \ip{t}^{-1/2}\mathcal{M}^\lambda_v[ h]\mathcal{E}^\lambda_v[\partial^2_v h]\!+\! \ip{t}^{-1}\!\left(\mathcal{M}^\lambda_v[\partial_v h]^2+\mathcal{E}^\lambda_x[\partial_x^2h]\mathcal{E}^\lambda_x[h]\right)
        \!+\!\ip{t}^{-\frac32} \mathcal{E}_x^\lambda[\partial_x h]^2.\quad
        \label{sec7:bound2N}.
          \end{align}
\end{lemma}
\begin{proof}
We begin by proving \eqref{sec7:bound1N}. Recalling that $\phi_h=G_1\ast\rho_h$ with the Green function $G_1$ from \eqref{green}, when $0\leq t\leq 1$ we directly obtain that
\begin{equation*}
    \mathcal{N}^\lambda_v[\ip{t}\partial_x\phi_h(t)]\lesssim \mathcal{E}^\lambda_x[\partial_xh]\mathcal{E}^\lambda_x[h].
\end{equation*}
We henceforth assume that $t\geq 1$. With the standard dispersive change of variables we have that
\begin{equation*}
    \begin{aligned}
        \phi_h(x,t)
        &=\iint G_1(y)h^2(x-y-tw,w,t)dydw=t^{-1}\iint G_1(y) h^2\left(a, \frac{x-y-a}{t},t\right)dyda
    \end{aligned}
    \end{equation*}
and thus
\begin{equation*}
\begin{aligned}
   t^{n+1}\partial_x^{n+1} \phi_h(x,t)=t^{-1}\sum_{i=0}^{n+1}\binom{n+1}{i}\int_0^\infty &
        \iint_{\R\times \R}\chi(r^{-\frac12}z) \partial_v^{i}h\left( a, \frac{x-z-a}{t},t\right)\\
        & \qquad \cdot\partial_v^{n+1-i}h\left( a, \frac{x-z-a}{t},t\right) 
        dz da 
        \frac{e^{-r}}{\sqrt{4\pi r}}dr. 
\end{aligned}        
\end{equation*}
It follows with the splitting $n+1=(n+1-i)+i$ that 
\begin{equation*}
        \frac{\lambda_t^n}{n!}t^{n+1}
        \partial_x^{n+1} \phi_h(x,t)= I_1+I_2,
\end{equation*}
where
\begin{equation*}
\begin{aligned}
    I_1&:= t^{-1}\sum_{i=1}^{n+1}
        \iint G_1(z) \frac{\lambda^{i-1}}{(i-1)!}\partial_v^{i}h\left( a, \frac{x-z-a}{t},t\right)
         \frac{\lambda_t^{n+1-i}}{(n+1-i)!}\partial_v^{n+1-i}h\left( a, \frac{x-z-a}{t},t\right) 
        dz da,\\
    I_2&:= t^{-1}\sum_{i=0}^{n}
        \iint G_1(z) \frac{\lambda^i}{i!}\partial_v^{i}h\left( a, \frac{x-z-a}{t},t\right)
         \frac{\lambda_t^{n-i}}{(n-i)!}\partial_v^{n+1-i}h\left( a, \frac{x-z-a}{t},t\right) 
        dz da.
\end{aligned}        
\end{equation*}
By taking the $\norm{\cdot}_{L^2_x L^\infty_v}$ norm in the terms in $I_1, I_2$ with the same amount of derivatives as in the corresponding factorial, we conclude that
\begin{align*}
    \sum_{n=0}^\infty &\frac{\lambda^n}{n!}\norm{\ip{t}^{n+1}\partial_x^{n+1}\phi_h(t)}_{L^\infty_x(\R)}
    \\
    &\leq
     t^{-\frac12} \sum_{n=0}^\infty \lambda^n \sum_{i=1}^{n+1}
    \frac{\|\partial_v^i h\|_{L^2_{x,v}}}{(i-1)!}
    \frac{\|\partial_v^{n+1-i} h\|_{L^2_x L^\infty_v}}{(n+1-i)!}
    +
     t^{-\frac12} \sum_{n=0}^\infty \lambda^n \sum_{i=0}^{n}
    \frac{\|\partial_v^i h\|_{L^2_x L^\infty_v}}{i!}
    \frac{\|\partial_v^{n+1-i} h\|_{L^2_{x,v}}}{(n-i)!}\\
    &\lesssim t^{-\frac12}\mathcal{M}^\lambda_v[h]\mathcal{E}^\lambda_v[\partial_vh],
\end{align*}
where in the last inequality we inverted the order in the summations. This implies the claim.

The estimate \eqref{sec7:bound2N} follows analogously, using for $0\leq t\leq 1$ that
\begin{equation*}
  \mathcal{N}^\lambda_v[\ip{t}^2\partial_x^2\phi_h(t)]\lesssim \mathcal{E}^\lambda_x[\partial_x^2h]\mathcal{E}^\lambda_x[h]+\mathcal{E}^\lambda_x[\partial_xh]^2  
\end{equation*}
and for $t\geq 1$ that
     \begin{align*}
        t^{n+2}\partial_x^{n+2} \phi_h(x,t)=t^{-1}\sum_{i=0}^{n+2}
        \binom{n+2}{i}
        \int_0^\infty &\iint_{\R\times \R}\chi(r^{-\frac12}z)  \partial_v^{i}h\left( a, \frac{x-z-a}{t},t\right)\\
        & \quad \cdot \partial_v^{n+2-i}h\left( a, \frac{x-z-a}{t},t\right) dz da \frac{e^{-r}}{\sqrt{4\pi r}}dr.
    \end{align*}
    \end{proof}

\subsection{Nonlinear analysis}\label{sec:1d_nonlin}

We establish the following bootstrap, which readily implies Theorem \ref{thm:1d}.

\begin{proposition}[Bootstrap estimates in $d=1$]
\label{sec7:boot}
There exists $\eps_0>0$ such that if $\gamma$ is a solution to \eqref{eq:gamma} satisfying the initial data assumptions \eqref{eq:1d_assumptions} and 
\begin{equation}
    \label{sec7:bootass}
        \sum_{a\in\{0,1\}}\left(\mathcal{E}^{\lambda_t}_v[\partial_v^a\gamma(t)]+ \mathcal{E}_x^{\lambda_t}[\partial_x^a\gamma(t)]\right)\lesssim \eps\leq \eps_0,\qquad 0\leq t\leq T,
\end{equation}
where $T\gtrsim R^2\eps^{-4}$, then we have the improved bound
\begin{equation*}
\sum_{a\in\{0,1\}}\left(\mathcal{E}^{\lambda_t}_v[\partial_v^a\gamma(t)]+ \mathcal{E}_x^{\lambda_t}[\partial_x^a\gamma(t)]\right) \le \eps + C \eps^3,\qquad 0\leq t\leq T. 
\end{equation*}
\end{proposition}

Since henceforth all our norms will depend on the same time-dependent analyticity parameter
\begin{equation*}
    \lambda=\lambda_t:=R-C\eps^2\ip{t}^{\frac12}
\end{equation*}
as in Proposition \ref{sec7:boot} and Theorem \ref{thm:1d}, hereafter we will suppress it to not overburden the notation.

\begin{proof}
We first notice that by the bootstrap assumptions \eqref{sec7:bootass}, with Lemma \ref{lem:1d_decay} we have the following decay bounds,
\begin{equation}
\label{sec7:boot:dec}
\begin{aligned}
        \mathcal{N}_v[\ip{t}\partial_x\phi(t)]&\lesssim \varepsilon^2\ip{t}^{-1/2}, \\
        \mathcal{N}_v[\ip{t}^2\partial_x^2\phi(t)]&\lesssim 
        \varepsilon \ip{t}^{-1/2}
        \left(\mathcal{E}_x[\partial_x^2\gamma(t)]+\mathcal{E}_v[\partial_v^2\gamma(t)]\right)+\varepsilon^2\ip{t}^{-\frac32}.
\end{aligned}
\end{equation}

\emph{I. Control of $\mathcal{E}_x[\gamma]+\mathcal{E}_v[\gamma]$.}
We differentiate
\begin{equation*}
    \mathcal{E}_x[\gamma(t)]+\mathcal{E}_v[\gamma(t)]= \sum_{n=0}^\infty \frac{\lambda^n_t}{n!}\left(
    \norm{\partial_x^n \gamma(t)}_{L^2_{x,v}}+
    \norm{\partial_v^n \gamma(t)}_{L^2_{x,v}}
    \right)
\end{equation*}
to get that
 \begin{align}\label{eq:dtEv1}
        \frac{d}{dt} \left(
        \mathcal{E}_x[\gamma(t)]+\mathcal{E}_v[\gamma(t)]\right)
      &=\dot\lambda_t \left(\mathcal{E}_x[\partial_x \gamma(t)]+\mathcal{E}_{v}[\partial_v \gamma(t)]\right)
      +\mathcal{A}_x + \mathcal{A}_v
      \end{align}
      with
      \begin{equation}
      \label{sec7:a1a2}
      \mathcal{A}_x:=\sum_{n=1}^\infty \frac{\lambda^n_t}{n!}\frac{d}{dt} \norm{\partial_x^{n}\gamma(t)}_{L^2_{x,v}},\qquad
      \mathcal{A}_v:=\sum_{n=1}^\infty \frac{\lambda^n_t}{n!}\frac{d}{dt} \norm{\partial_v^{n}\gamma(t)}_{L^2_{x,v}},
        \end{equation}
        where we used that $\norm{\gamma}_{L^2_{x,v}}$ is conserved. 
        
        Concerning $\mathcal{A}_x$ in \eqref{sec7:a1a2},

        from the equation \eqref{eq:gamma} for $\gamma$ we directly obtain that
    \begin{align*}
        \frac12\frac{d}{dt} \norm{\partial_x^{n} \gamma}^2_{L^2_{x,v}}&=\sum_{\substack{j=1}}^{n}
        \binom{n}{j}
        \iint 
        \partial_x^{n}\gamma \partial_x^{j+1}\phi(x+tv,t)\partial_v\partial_x^{n-j} \gamma dx dv\\
        &\qquad -t\sum_{\substack{j=1}}^{n}
        \binom{n}{j}\iint 
        \partial_x^{n}\gamma \partial_x^{j+1}\phi(x+tv,t)\partial_x^{n-j+1} \gamma dx dv\\
        &=:\mathcal{B}_1+\mathcal{B}_2,
        \end{align*}
  with bounds
        \begin{align}
        \label{sec7:stimax1}
        \mathcal{B}_1&\le n! \norm{\partial_x^{n}\gamma}_{L^2_{x,v}} \sum_{j=1}^n \frac{\norm{\partial_x^{j+1}\phi}_{L^\infty_x}}{j!}
        \frac{\norm{\partial_v \partial_x^{n-j}\gamma}_{L^2_{x,v}}}{(n-j)!},
    \end{align}
and 
    \begin{align}
        \label{sec7:stimax2}
        \mathcal{B}_2&\le t n! \norm{\partial_x^{n}\gamma}_{L^2_{x,v}} \sum_{j=1}^n \frac{\norm{\partial_x^{j+1}\phi}_{L^\infty_x}}{j!}
        \frac{\norm{\partial_x^{n-j+1}\gamma}_{L^2_{x,v}}}{(n-j)!}.
    \end{align}
    Collecting \eqref{sec7:stimax1} and \eqref{sec7:stimax2}, we arrive at
    \begin{align*}
        \mathcal{A}_x
        &\le 
        \left(\sum_{j=1}^\infty \frac{\lambda_t^j}{j!} \norm{\partial_x^{j+1}\phi}_{L^\infty_x}\right)
        \left(\sum_{n=0}^\infty \frac{\lambda_t^{n}}{n!}\norm{\partial_v \partial_x^n \gamma}_{L^2_{x,v}}\right) \\
        &\qquad + t \left(\sum_{j=1}^\infty \frac{\lambda_t^j}{j!} \norm{\partial_x^{j+1}\phi}_{L^\infty_x}\right)
        \left(\sum_{n=0}^\infty \frac{\lambda_t^{n}}{n!}\norm{ \partial_x^{n+1} \gamma}_{L^2_{x,v}}\right)\\
        &\leq \ip{t}^{-2}\mathcal{N}_v[\ip{t}\partial_x\phi(t)]\mathcal{E}_x[\partial_v \gamma(t)]+\ip{t}^{-1}\mathcal{N}_v[\ip{t}\partial_x\phi(t)]\mathcal{E}_x[\partial_x \gamma(t)].
    \end{align*}
    With the interpolation estimates in \eqref{sec7:interp} and the decay bounds \eqref{sec7:boot:dec}, this yields
    \begin{equation}
    \label{sec7:0xvpart1}
        \mathcal{A}_x \lesssim
        \left(\ip{t}^{-\frac52}+
        \ip{t}^{-\frac32}
        \right)\varepsilon^3.
    \end{equation}
    
We next study $\mathcal{A}_v$. We have
    \begin{align*}
        \frac12\frac{d}{dt} \norm{\partial_v^{n} \gamma}^2_{L^2_{x,v}}&=\sum_{\substack{j=1}}^{n}
        \binom{n}{j}
        \iint 
        \partial_v^{n}\gamma t^j\partial_x^{j+1}\phi(x+tv,t)\partial_v^{n-j+1} \gamma dx dv\\
        &\qquad -\sum_{\substack{j=1}}^{n}
        \binom{n}{j}\iint 
        \partial_v^{n}\gamma (t\partial_x)^{j+1}\phi(x+tv,t)\partial_x\partial_v^{n-j} \gamma dx dv\\
        &=:\mathcal{C}_1+\mathcal{C}_2,
        \end{align*}
  with bounds
        \begin{align}
        \label{sec7:estv}
        \mathcal{C}_1&\le n! \norm{\partial_v^{n}\gamma}_{L^2_{x,v}} \sum_{j=1}^n \frac{\norm{\ip{t}^j\partial_x^{j+1}\phi}_{L^\infty_x}}{j!}
        \frac{\norm{\partial_v^{n-j+1}\gamma}_{L^2_{x,v}}}{(n-j)!},
    \end{align}
and 
    \begin{align}
        \label{sec7:estv2}
        \mathcal{C}_2&\le n! \norm{\partial_v^{n}\gamma}_{L^2_{x,v}} \sum_{j=1}^n \frac{\norm{\ip{t}^{j+1}\partial_x^{j+1}\phi}_{L^\infty_x}}{j!}
        \frac{\norm{\partial_x\partial_v^{n-j}\gamma}_{L^2_{x,v}}}{(n-j)!}.
    \end{align}
    Collecting \eqref{sec7:estv} and \eqref{sec7:estv2}, we arrive at
    \begin{align*}
       \mathcal{A}_v
        &\le 
        \left(\sum_{j=1}^\infty \frac{\lambda_t^j}{j!} \norm{\ip{t}^j\partial_x^{j+1}\phi}_{L^\infty_x}\right)
        \left(\sum_{n=0}^\infty \frac{\lambda_t^{n}}{n!}\norm{\partial_v^{n+1} \gamma}_{L^2_{x,v}}\right) \\
        &\qquad + \left(\sum_{j=1}^\infty \frac{\lambda_t^j}{j!} \norm{\ip{t}^{j+1}\partial_x^{j+1}\phi}_{L^\infty_x}\right)
        \left(\sum_{n=0}^\infty \frac{\lambda_t^{n}}{n!}\norm{ \partial_x\partial_v^{n} \gamma}_{L^2_{x,v}}\right)\\
        &\leq \ip{t}^{-1}\mathcal{N}_v[\ip{t}\partial_x\phi(t)]\mathcal{E}_v[\partial_v \gamma(t)]+\mathcal{N}_v[\ip{t}\partial_x\phi(t)]\mathcal{E}_v[\partial_x \gamma(t)].
    \end{align*}

By the decay bounds \eqref{sec7:boot:dec} and the interpolation estimate 
    \begin{align*}
        \mathcal{E}_v[\partial_x \gamma]\le\mathcal{E}_v[\partial_v \gamma]+\mathcal{E}_x[ \partial_x\gamma],
    \end{align*}
we conclude
\begin{equation}
\label{sec7:0xvpart2}
    \mathcal{A}_v \lesssim \ip{t}^{-\frac32}\varepsilon^3 + \ip{t}^{-\frac12}\varepsilon^2\left(
    \mathcal{E}_v[\partial_v \gamma(t)]+\mathcal{E}_x[ \partial_x\gamma(t)]
    \right).
\end{equation}

Inserting \eqref{sec7:0xvpart1} and \eqref{sec7:0xvpart2} in \eqref{eq:dtEv1}, there holds that
\begin{equation*}
   \frac{d}{dt} \left(\mathcal{E}_x[ \gamma]+\mathcal{E}_v[ \gamma]\right)
      \le(\dot\lambda_t+C\eps^2\ip{t}^{-\frac12}) \left(\mathcal{E}_{x}[\partial_x \gamma]+ \mathcal{E}_{v}[\partial_v \gamma]\right)
      + C\varepsilon^3 \ip{t}^{-\frac32}
      \le C\varepsilon^3 \ip{t}^{-\frac32}
\end{equation*}
upon choosing $\lambda_t=R-C\eps^2\ip{t}^{\frac12}$ as in the statement. This gives the claim.

\medskip
\emph{II. Control of $\mathcal{E}_x[\partial_x \gamma]+\mathcal{E}_v[\partial_v \gamma]$.}

We now study 
\begin{equation*}
    \mathcal{E}_x[\partial_x \gamma]+\mathcal{E}_v[\partial_v \gamma]= \sum_{n=0}^\infty \frac{\lambda^n_t}{n!}
    \left(
    \norm{\partial_x^{n+1} \gamma}_{L^2_{x,v}}
    +
    \norm{\partial_v^{n+1} \gamma}_{L^2_{x,v}}
    \right),
\end{equation*}
for which
 \begin{align}
 \label{sec7:eqdxg}
        \frac{d}{dt} \left(
        \mathcal{E}_x[\partial_x \gamma]
        +
        \mathcal{E}_v[\partial_v \gamma]
        \right)
      &=\dot\lambda_t \left(
      \mathcal{E}_{x}[\partial^2_x \gamma]
      + \mathcal{E}_{v}[\partial_v^2 \gamma]
      \right)+\mathcal{T}_x + \mathcal{T}_v
        \end{align}
with
     \begin{equation}
     \label{sec7:t1t2}
        \mathcal{T}_x:=\sum_{n=0}^\infty \frac{\lambda^n_t}{n!}
      \frac{d}{dt} \norm{\partial_x^{n+1}\gamma}_{L^2_{x,v}},\qquad \mathcal{T}_v:=\sum_{n=0}^\infty \frac{\lambda^n_t}{n!}
      \frac{d}{dt} \norm{\partial_v^{n+1}\gamma}_{L^2_{x,v}}. 
     \end{equation}
     
    Concerning $\mathcal{T}_x$ in \eqref{sec7:t1t2}, we get
        \begin{align*}
        \frac12\frac{d}{dt} \norm{\partial_x^{n+1} \gamma}^2_{L^2_{x,v}}&=\sum_{\substack{j=1}}^{n+1}
        \binom{n+1}{j}
        \iint 
        \partial_x^{n+1}\gamma \partial_x^{j+1}\phi(x+tv,t)\partial_v\partial_x^{n+1-j} \gamma dx dv\\
        &\qquad+t\sum_{\substack{j=1}}^{n+1}
        \binom{n+1}{j}
        \iint 
        \partial_x^{n+1}\gamma \partial_x^{j+1}\phi(x+tv,t)\partial_x^{n+2-j} \gamma dx dv.
        \end{align*}
    It follows that
\begin{align*}
        \sum_{n=0}^\infty 
        \frac{\lambda^n_t}{n!}\frac{d}{dt} \norm{\partial_x^{n+1}\gamma}_{L^2_{x,v}} 
        &\le 
        \ip{t}^{-2}\left(\sum_{j=0}^\infty \frac{\lambda_t^{j}}{j!} \norm{\ip{t}^{j+2}\partial_x^{j+2}\phi}_{L^\infty_x}\right)
        \left(\sum_{n=0}^\infty \frac{\lambda_t^{n}}{n!}\norm{\partial_v \partial_x^n \gamma}_{L^2_{x,v}}\right) \\
        &\qquad + \ip{t}^{-1}\left(\sum_{j=0}^\infty \frac{\lambda_t^j}{j!} \norm{\ip{t}^{j+1}\partial_x^{j+1}\phi}_{L^\infty_x}\right)
        \left(\sum_{n=0}^\infty \frac{\lambda_t^{n}}{n!}\norm{\partial_v \partial_x^{n+1} \gamma}_{L^2_{x,v}}\right) \\
        &\qquad +  \ip{t}^{-1}\left(\sum_{j=0}^\infty \frac{\lambda_t^j}{j!} \norm{\ip{t}^{j+2}\partial_x^{j+2}\phi}_{L^\infty_x}\right)
        \left(\sum_{n=0}^\infty \frac{\lambda_t^{n}}{n!}\norm{ \partial_x^{n+1} \gamma}_{L^2_{x,v}}\right) \\
        &\qquad+ \left(\sum_{j=0}^\infty \frac{\lambda_t^j}{j!} \norm{\ip{t}^{j+1}\partial_x^{j+1}\phi}_{L^\infty_x}\right)
        \left(\sum_{n=0}^\infty \frac{\lambda_t^{n}}{n!}\norm{ \partial_x^{n+2} \gamma}_{L^2_{x,v}}\right)\\
        &=:B_1+B_2+B_3+B_4.
    \end{align*}

    By the interpolation estimates in \eqref{sec7:interp} and \eqref{sec7:interp1}, we get
    \begin{align*}
        B_1+B_2 \le \ip{t}^{-2}\mathcal{N}_v[\ip{t}^2\partial_x^2 \phi]\left(
        \mathcal{E}_x[\partial_x \gamma]+\mathcal{E}_v[\partial_v \gamma]
        \right)
        + \ip{t}^{-1}\mathcal{N}_v[\ip{t}\partial_x\phi]\left(
        \mathcal{E}_x[\partial_x^2 \gamma] + \mathcal{E}_v[\partial^2_v \gamma]
        \right),
    \end{align*}
    while
    \begin{align*}
        B_3 + B_4 \le  \ip{t}^{-1}\mathcal{N}_v[\ip{t}^2\partial_x^2 \phi]\mathcal{E}_x[\partial_x \gamma] + \mathcal{N}_v[\ip{t}\partial_x\phi]\mathcal{E}_x[\partial_x^2 \gamma],
    \end{align*}
    so that by \eqref{sec7:boot:dec} and \eqref{sec7:bootass}
    \begin{equation}\label{eq:1dT1bd}
      \mathcal{T}_x\lesssim \left(\ip{t}^{-\frac12}+\ip{t}^{-\frac32}+\ip{t}^{-\frac52}\right)\eps^2\left(\mathcal{E}_v[\partial_v^2\gamma(t)]+\mathcal{E}_x[\partial_x^2\gamma(t)]\right)+\ip{t}^{-\frac32}\eps^3.
    \end{equation}

    We next study $\mathcal{T}_v$ in \eqref{sec7:t1t2}. With
    \begin{align*}
        \frac12\frac{d}{dt} \norm{\partial_v^{n+1} \gamma}^2_{L^2_{x,v}}&=\sum_{\substack{j=1}}^{n+1}
        \binom{n+1}{j}\iint 
        \partial_v^{n+1}\gamma t^{j}\partial_x^{j+1}\phi(x+tv,t)\partial_v^{n+2-j} \gamma dx dv\\
        &\qquad -\sum_{\substack{j=1}}^{n+1}
        \binom{n+1}{j}
        \iint 
        \partial_v^{n+1}\gamma (t\partial_x)^{j+1}\phi(x+tv,t)\partial_x \partial_v^{n+1-j} \gamma dx dv\\
        &=:C_1+C_2,
        \end{align*}
    we obtain that
        \begin{align*}
        C_1&\le \ip{t}^{-1}(n+1)! \norm{\partial_v^{n+1}\gamma}_{L^2_{x,v}} \sum_{j=1}^{n+1} \frac{\norm{\ip{t}^{j+1}\partial_x^{j+1}\phi}_{L^\infty_x}}{j!}
        \frac{\norm{\partial_v^{n+2-j}\gamma}_{L^2_{x,v}}}{(n+1-j)!},
    \end{align*}
    and
    \begin{align*}
        C_2 \le (n+1)! \norm{\partial_v^{n+1}\gamma}_{L^2_{x,v}}
        \sum_{j=1}^{n+1} \frac{\norm{\ip{t}^{j+1}\partial_x^{j+1}\phi}_{L^\infty_x}}{j!} \frac{\norm{\partial_x \partial_v^{n+1-j}}_{L^2_{x,v}}}{(n+1-j)!},
    \end{align*}
so that altogether
     \begin{align*}
        \sum_{n=0}^\infty 
        \frac{\lambda^n_t}{n!}\frac{d}{dt} \norm{\partial_v^{n+1}\gamma}_{L^2_{x,v}} 
        &\le 
        \ip{t}^{-1}\left(\sum_{j=0}^\infty \frac{\lambda_t^{j}}{j!} \norm{\ip{t}^{j+2}\partial_x^{j+2}\phi}_{L^\infty_x}\right)
        \left(\sum_{n=0}^\infty \frac{\lambda_t^{n}}{n!}\norm{\partial_v^{n+1} \gamma}_{L^2_{x,v}}\right) \\
        &\qquad + \ip{t}^{-1}\left(\sum_{j=0}^\infty \frac{\lambda_t^j}{j!} \norm{\ip{t}^{j+1}\partial_x^{j+1}\phi}_{L^\infty_x}\right)
        \left(\sum_{n=0}^\infty \frac{\lambda_t^{n}}{n!}\norm{\partial_v^{n+2} \gamma}_{L^2_{x,v}}\right) \\
        &\qquad + \left(\sum_{j=0}^\infty \frac{\lambda_t^j}{j!} \norm{\ip{t}^{j+2}\partial_x^{j+2}\phi}_{L^\infty_x}\right)
        \left(\sum_{n=0}^\infty \frac{\lambda_t^{n}}{n!}\norm{ \partial_x\partial_v^{n} \gamma}_{L^2_{x,v}}\right) \\
        &\qquad + \left(\sum_{j=0}^\infty \frac{\lambda_t^j}{j!} \norm{\ip{t}^{j+1}\partial_x^{j+1}\phi}_{L^\infty_x}\right)
        \left(\sum_{n=0}^\infty \frac{\lambda_t^{n}}{n!}\norm{ \partial_x\partial_v^{n+1} \gamma}_{L^2_{x,v}}\right)\\
        &=:D_1+D_2+D_3+D_4.
    \end{align*}
   We have
    \begin{align}
    \label{sec7:d1}
        D_1 + D_2 \le \ip{t}^{-1}\left(\mathcal{N}_v[\ip{t}^2\partial_x^2\phi]\mathcal{E}_v[\partial_v\gamma]+\mathcal{N}_v[\ip{t}\partial_x\phi]\mathcal{E}_v[\partial_v^2\gamma]\right),
    \end{align}
    while, by interpolation as in \eqref{sec7:interp} and \eqref{sec7:interp1} 
    \begin{align*}
        \mathcal{E}_v[\partial_x \gamma]\le\mathcal{E}_v[\partial_v \gamma]+\mathcal{E}_x[ \partial_x\gamma],\qquad \mathcal{E}_v[\partial_x\partial_v \gamma]\le \mathcal{E}_v[\partial^2_v \gamma]+
        \mathcal{E}_x[\partial^2_x\gamma],
    \end{align*}
    and we thus get
    \begin{align}
    \label{sec7:d2}
        D_3 + D_4 \le \mathcal{N}_v[\ip{t}^2\partial_x^2\phi](\mathcal{E}_x[\partial_x\gamma]+\mathcal{E}_v[\partial_v\gamma])+\mathcal{N}_v[\ip{t}\partial_x\phi](\mathcal{E}_x[\partial^2_x\gamma]+\mathcal{E}_v[\partial_v^2\gamma]).
    \end{align}
    Collecting \eqref{sec7:d1} and \eqref{sec7:d2} we obtain that
    \begin{equation}\label{eq:1dT2bd}
      \mathcal{T}_v\lesssim \left(\ip{t}^{-\frac12}+\ip{t}^{-\frac32}\right)\eps^2\left(\mathcal{E}_v[\partial_v^2\gamma]+\mathcal{E}_x[\partial^2_x\gamma]\right)+\ip{t}^{-\frac32}\varepsilon^3.
    \end{equation}
    Inserting \eqref{eq:1dT1bd} and \eqref{eq:1dT2bd} in \eqref{sec7:eqdxg} shows that for a universal constant $C>0$ there holds that
\begin{equation*}
   \frac{d}{dt} \left(\mathcal{E}_x[\partial_x \gamma]+\mathcal{E}_v[\partial_v \gamma]\right)
      \le(\dot\lambda_t+C\eps^2\ip{t}^{-\frac12}) \left(\mathcal{E}_{x}[\partial^2_x \gamma]+ \mathcal{E}_{v}[\partial_v^2 \gamma]\right)
      +C\ip{t}^{-\frac32}\varepsilon^3
      \le C \ip{t}^{-\frac32}\varepsilon^3
\end{equation*}
upon choosing $\lambda_t=R-C\eps^2\ip{t}^{\frac12}$ as in the statement, whereby the requirement that $\lambda_t>0$ yields the claimed time scale $T\gtrsim R^2\eps^{-4}$.    
\end{proof}

\addtocontents{toc}{\protect\setcounter{tocdepth}{0}}
\section*{Acknowledgments}
The authors would like to thank J.\ Smulevici for comments on a first version of this manuscript. K.\ Widmayer and S.\ Rossi gratefully acknowledge support of the SNSF through the grant PCEFP2\_203059.
M.\ Iacobelli acknowledges the support of the National Science Foundation through the grant No.\ DMS-1926686.
All authors acknowledge the support of the NCCR SwissMAP.

\addtocontents{toc}{\protect\setcounter{tocdepth}{1}}

\appendix
\section{Proof of an auxiliary statement}\label{sec:appdx}
\begin{proof}[Proof of \eqref{sec3:PL}]
For $h:\R^2\times\R^2\to\R$ we introduce the Littlewood-Paley decomposition
\begin{equation*}
h_M(x,v):=\mathcal{F}^{-1}_{\eta \to v}\left[\psi\left(M^{-1}\eta\right) \mathcal{F}_{v \to \eta}h\right](x,v),
\end{equation*}
where $\mathcal{F}$ is the Fourier transform in the velocity variable, $M \in 2^{\mathbb{Z}}$ a dyadic number and $\psi$ a Littlewood-Paley function, i.e.\ a bump function compactly supported in $2^{-1} \le |\eta| \le 2$ (see e.g. \cite[Appendix A]{T06}). We see that with $h = \sum_M h_M$ and
\begin{equation}
\label{sec3:Paley}
h_M\Big(a, \frac{x-a}{t}\Big) - h_M\Big(a, \frac{x}{t}\Big)
= \sum_{j=1}^2\int_0^1  
\frac{a_j}{t} \partial_{v_j}h_M\Big(a, \frac{x}{t}-\theta \frac{a}{t}\Big) d\theta,
\end{equation}
there holds that
\begin{align}
\label{sec3:Paley2}
     \Big\|
     h\Big(a,\frac{x-a}{t}\Big) - h\Big(a, \frac{x}{t}\Big)
     \Big\|_{ L^2_aL^\infty_x} 
     &\lesssim 
     \sum_{M \in 2^{\Z}} 
     \min
     \left\{
     t^{-1} M^2 \|x h \|_{L^2_{x,v}}, M^{-1}\|h\|_{L^2_x H^2_v}
     \right\},
\end{align}
where the first term inside the minimum is obtained from the r.h.s of \eqref{sec3:Paley} using that 
\begin{equation*}
\Bigg \|\sum_{j=1}^2\int_0^1  
\frac{a_j}{t} \partial_{v_j}h_M\Big(a, \frac{x}{t}-\theta \frac{a}{t}\Big) d\theta\Bigg\|_{L^2_aL^\infty_x }
\lesssim
t^{-1}\| x \partial_{v}h_M\|_{L^2_xL^\infty_v }
\end{equation*}
and then using Bernstein's inequalities 
$\|h_M\|_{L^2_xL^\infty_v } \lesssim M \|h_M\|_{L^2_{x,v}}$
and $\|\nabla_v h_M\|_{L^2_{x,v}} \simeq M \|h_M\|_{L^2_{x,v}}$ (see e.g. \cite{T06}).
The second term inside the minimum in \eqref{sec3:Paley2}
is obtained from the l.h.s of \eqref{sec3:Paley} by using again that $\|h_M\|_{L^2_xL^\infty_v }\lesssim M \|h_M\|_{L^2_{x,v}}\approx M^{-1}\|h_M\|_{L^2_x H^2_v}$.\\
For an integer $B>0$ to be fixed, we get 
\begin{equation*}
 \sum_{M \in 2^{\Z}} 
     \min
     \Big\{
     t^{-1} M^2 \|x h \|_{L^2_{x,v}}, M^{-1}\|h\|_{L^2_x H^2_v}
     \Big\}
     \lesssim 
     t^{-1} \|xh\|_{L^2_{x,v}}\sum_{j \le B}  2^{2j}
    + \|h\|_{L^2_x H^2_v}\sum_{j >B} 2^{-j}.
\end{equation*}
Optimizing in $B$, we conclude \eqref{sec3:PL}.  
\end{proof}

\bibliographystyle{siam}
\bibliography{bibliography} 

@book{Davidson1972,
  author    = {R. C. Davidson},
  title     = {Methods in Nonlinear Plasma Theory},
  publisher = {Academic Press},
  year      = {1972},
  pages     = {356},
  chapter   = {3}
}

@book{Chen2016,
  author    = {F. F. Chen},
  title     = {Introduction to Plasma Physics and Controlled Fusion},
  publisher = {Springer},
  year      = {2016},
  pages     = {421},
  chapter   = {2},
  note      = {Section 2.3}
}

@book{BinneyTremaine2008,
  author    = {J. Binney and S. Tremaine},
  title     = {Galactic Dynamics},
  publisher = {Princeton University Press},
  year      = {2008},
  pages     = {920},
  chapter   = {5}
}

@article{Vogelsberger2014,
  author    = {M. Vogelsberger and et al.},
  title     = {Introducing the Illustris Project: Simulating the Coevolution of Dark and Visible Matter in the Universe},
  journal   = {Monthly Notices of the Royal Astronomical Society},
  volume    = {444},
  number    = {2},
  pages     = {1518--1547},
  year      = {2014}
}

@Article{IPWW2020,
  author   = {A. D. Ionescu and B. Pausader and X. Wang and K. Widmayer},
  journal  = {International Mathematics Research Notices. IMRN},
  title    = {On the {A}symptotic {B}ehavior of {S}olutions to the {V}lasov--{P}oisson {S}ystem},
  year     = {2022},
  issn     = {1073-7928},
  number   = {12},
  pages    = {8865--8889},
  doi      = {10.1093/imrn/rnab155},
  mrnumber = {4436197},
}

@article {N72,
    AUTHOR = {Nirenberg, L.},
     TITLE = {An abstract form of the nonlinear {C}auchy--{K}owalewski
              theorem},
   JOURNAL = {J. Differential Geometry},
  FJOURNAL = {Journal of Differential Geometry},
    VOLUME = {6},
      YEAR = {1972},
     PAGES = {561--576},
      ISSN = {0022-040X,1945-743X},
   MRCLASS = {35A10 (35R20)},
  MRNUMBER = {322321},
MRREVIEWER = {W.\ Watzlawek},
       URL = {http://projecteuclid.org/euclid.jdg/1214430643},
}

@article {GI23,
    AUTHOR = {Gagnebin, Antoine and Iacobelli, Mikaela},
     TITLE = {Landau damping on the torus for the {V}lasov--{P}oisson system
              with massless electrons},
   JOURNAL = {J. Differential Equations},
  FJOURNAL = {Journal of Differential Equations},
    VOLUME = {376},
      YEAR = {2023},
     PAGES = {154--203},
      ISSN = {0022-0396,1090-2732},
   MRCLASS = {35Q82 (82D10)},
  MRNUMBER = {4637255},
MRREVIEWER = {Assane\ Lo},
       DOI = {10.1016/j.jde.2023.08.020},
       URL = {https://doi.org/10.1016/j.jde.2023.08.020},
}

@article {GPI21,
    AUTHOR = {Griffin-Pickering, Megan and Iacobelli, Mikaela},
     TITLE = {Global strong solutions in {$\Bbb R^3$} for ionic
              {V}lasov--{P}oisson systems},
   JOURNAL = {Kinet. Relat. Models},
  FJOURNAL = {Kinetic and Related Models},
    VOLUME = {14},
      YEAR = {2021},
    NUMBER = {4},
     PAGES = {571--597},
      ISSN = {1937-5093,1937-5077},
   MRCLASS = {35Q83 (35A01 35Q82 82C40 82C70 82D10)},
  MRNUMBER = {4296179},
       DOI = {10.3934/krm.2021016},
       URL = {https://doi.org/10.3934/krm.2021016},
}

@article {HNX24,
AUTHOR = {Lingjia Huang and Quoc-Hung Nguyen and Yiran Xu},
JOURNAL = {Kinet. Relat. Models},
TITLE = {Sharp estimates for screened {V}lasov--{P}oisson system around {P}enrose-stable equilibria in $ \mathbb{R}^d  $, $  d\geq3 $},
pages = {},
YEAR = {2024},
ISSN = {1937-5093},
DOI = {10.3934/krm.2024015},
URL = {https://www.aimsciences.org/article/id/667157485a42b314c5bff8bf},
}

@article {N77,
    AUTHOR = {Nishida, Takaaki},
     TITLE = {A note on a theorem of {N}irenberg},
   JOURNAL = {J. Differential Geometry},
  FJOURNAL = {Journal of Differential Geometry},
    VOLUME = {12},
      YEAR = {1977},
    NUMBER = {4},
     PAGES = {629--633},
      ISSN = {0022-040X,1945-743X},
   MRCLASS = {58E05 (35A10)},
  MRNUMBER = {512931},
MRREVIEWER = {J.\ Chrastina},
       URL = {http://projecteuclid.org/euclid.jdg/1214434231},
}

@Article {FOPW2024,
    AUTHOR = {Flynn, Patrick and Ouyang, Zhimeng and Pausader, Benoit and
              Widmayer, Klaus},
     TITLE = {Scattering map for the {V}lasov--{P}oisson system},
   JOURNAL = {Peking Math. J.},
  FJOURNAL = {Peking Mathematical Journal},
    VOLUME = {6},
      YEAR = {2023},
    NUMBER = {2},
     PAGES = {365--392},
      ISSN = {2096-6075,2524-7182},
   MRCLASS = {35Q83 (35B40 35Q70)},
  MRNUMBER = {4619597},
       DOI = {10.1007/s42543-021-00041-x},
       URL = {https://doi.org/10.1007/s42543-021-00041-x},
}

@Article {BGNS2018,
    AUTHOR = {Bardos, Claude and Golse, Francois and Nguyen, Toan T. and
              Sentis, R\'emi},
     TITLE = {The {M}axwell--{B}oltzmann approximation for ion kinetic
              modeling},
   JOURNAL = {Phys. D},
  FJOURNAL = {Physica D. Nonlinear Phenomena},
    VOLUME = {376/377},
      YEAR = {2018},
     PAGES = {94--107},
      ISSN = {0167-2789,1872-8022},
   MRCLASS = {82D10 (35Q20 35Q83)},
  MRNUMBER = {3815207},
MRREVIEWER = {Calvin\ Tadmon},
       DOI = {10.1016/j.physd.2017.10.014},
       URL = {https://doi.org/10.1016/j.physd.2017.10.014},
}

@incollection {GPI2021,
    AUTHOR = {Griffin-Pickering, Megan and Iacobelli, Mikaela},
     TITLE = {Recent developments on the well--posedness theory for
              {V}lasov--type equations},
 BOOKTITLE = {From particle systems to partial differential equations},
    SERIES = {Springer Proc. Math. Stat.},
    VOLUME = {352},
     PAGES = {301--319},
 PUBLISHER = {Springer, Cham},
      YEAR = {2021},
      ISBN = {978-3-030-69783-9; 978-3-030-69784-6},
   MRCLASS = {82D10 (35Q82 76X05)},
  MRNUMBER = {4376190},
       DOI = {10.1007/978-3-030-69784-6\_14},
       URL = {https://doi.org/10.1007/978-3-030-69784-6_14},
}

@article{BAMP2024,
    AUTHOR = {Ben-Artzi, Jonathan and Morisse, Baptiste and Pankavich,
              Stephen},
     TITLE = {Asymptotic growth and decay of two--dimensional symmetric
              plasmas},
   JOURNAL = {Kinet. Relat. Models},
  FJOURNAL = {Kinetic and Related Models},
    VOLUME = {17},
      YEAR = {2024},
    NUMBER = {1},
     PAGES = {29--51},
      ISSN = {1937-5093,1937-5077},
   MRCLASS = {35Q83 (35B40)},
  MRNUMBER = {4701222},
       DOI = {10.3934/krm.2023015},
       URL = {https://doi.org/10.3934/krm.2023015},
}

@article {IPWW2024,
    AUTHOR = {A. D. Ionescu and B. Pausader and X. Wang and K. Widmayer},
     TITLE = {Nonlinear {L}andau damping for the {V}lasov--{P}oisson system
              in {$\Bbb R^3$}: the {P}oisson equilibrium},
   JOURNAL = {Ann. PDE},
  FJOURNAL = {Annals of PDE. Journal Dedicated to the Analysis of Problems
              from Physical Sciences},
    VOLUME = {10},
      YEAR = {2024},
    NUMBER = {1},
     PAGES = {Paper No. 2, 78},
      ISSN = {2524-5317,2199-2576},
   MRCLASS = {35F50 (35B35)},
  MRNUMBER = {4677763},
       DOI = {10.1007/s40818-023-00161-w},
       URL = {https://doi.org/10.1007/s40818-023-00161-w},
}

@article{IPWW2025,
      title={Nonlinear {L}andau damping and wave operators in sharp {G}evrey spaces}, 
      author={A. D. Ionescu and B. Pausader and X. Wang and K. Widmayer},
      year={2024},
      note={\url{https://arxiv.org/abs/2405.04473}} 
}

@article {MV11,
    AUTHOR = {Mouhot, Cl\'ement and Villani, C\'edric},
     TITLE = {On {L}andau damping},
   JOURNAL = {Acta Math.},
  FJOURNAL = {Acta Mathematica},
    VOLUME = {207},
      YEAR = {2011},
    NUMBER = {1},
     PAGES = {29--201},
      ISSN = {0001-5962,1871-2509},
   MRCLASS = {82D10 (82C05)},
  MRNUMBER = {2863910},
       DOI = {10.1007/s11511-011-0068-9},
       URL = {https://doi.org/10.1007/s11511-011-0068-9},
}

@article {GNR2021,
    AUTHOR = {Grenier, Emmanuel and Nguyen, Toan T. and Rodnianski, Igor},
     TITLE = {Landau damping for analytic and {G}evrey data},
   JOURNAL = {Math. Res. Lett.},
  FJOURNAL = {Mathematical Research Letters},
    VOLUME = {28},
      YEAR = {2021},
    NUMBER = {6},
     PAGES = {1679--1702},
      ISSN = {1073-2780,1945-001X},
   MRCLASS = {35Q83 (42B37)},
  MRNUMBER = {4477671},
MRREVIEWER = {M.\ Pirner},
       DOI = {10.4310/mrl.2021.v28.n6.a3},
       URL = {https://doi.org/10.4310/mrl.2021.v28.n6.a3},
}

@article {BMM2016,
    AUTHOR = {Bedrossian, Jacob and Masmoudi, Nader and Mouhot, Cl\'ement},
     TITLE = {Landau damping: paraproducts and {G}evrey regularity},
   JOURNAL = {Ann. PDE},
  FJOURNAL = {Annals of PDE. Journal Dedicated to the Analysis of Problems
              from Physical Sciences},
    VOLUME = {2},
      YEAR = {2016},
    NUMBER = {1},
     PAGES = {Art. 4, 71},
      ISSN = {2524-5317,2199-2576},
   MRCLASS = {82D10 (35A20 35B35 35Q83 82C05)},
  MRNUMBER = {3489904},
MRREVIEWER = {Xianwen\ Zhang},
       DOI = {10.1007/s40818-016-0008-2},
       URL = {https://doi.org/10.1007/s40818-016-0008-2},
}

@article {HKNR21,
    AUTHOR = {Han-Kwan, Daniel and Nguyen, Toan T. and Rousset,
              Fr\'ed\'eric},
     TITLE = {Asymptotic stability of equilibria for screened
              {V}lasov--{P}oisson systems via pointwise dispersive estimates},
   JOURNAL = {Ann. PDE},
  FJOURNAL = {Annals of PDE. Journal Dedicated to the Analysis of Problems
              from Physical Sciences},
    VOLUME = {7},
      YEAR = {2021},
    NUMBER = {2},
     PAGES = {Paper No. 18, 37},
      ISSN = {2524-5317,2199-2576},
   MRCLASS = {35Q83 (76E30 82D10)},
  MRNUMBER = {4303653},
       DOI = {10.1007/s40818-021-00110-5},
       URL = {https://doi.org/10.1007/s40818-021-00110-5},
}

@article {BMM18,
    AUTHOR = {Bedrossian, Jacob and Masmoudi, Nader and Mouhot, Cl\'ement},
     TITLE = {Landau damping in finite regularity for unconfined systems
              with screened interactions},
   JOURNAL = {Comm. Pure Appl. Math.},
  FJOURNAL = {Communications on Pure and Applied Mathematics},
    VOLUME = {71},
      YEAR = {2018},
    NUMBER = {3},
     PAGES = {537--576},
      ISSN = {0010-3640,1097-0312},
   MRCLASS = {35Q83 (76X05 82C40 82D10)},
  MRNUMBER = {3762277},
MRREVIEWER = {Calvin\ Tadmon},
       DOI = {10.1002/cpa.21730},
       URL = {https://doi.org/10.1002/cpa.21730},
}

@article{HHX22,
 author = {Huang, Lingjia and Nguyen, Quoc-Hung and Xu, Yiran},
 title = {Nonlinear {Landau} damping for the 2D {Vlasov}-{Poisson} system with massless electrons around {Penrose}-stable equilibrium},
 fjournal = {SIAM Journal on Mathematical Analysis},
 journal = {SIAM J. Math. Anal.},
 issn = {0036-1410},
 volume = {57},
 number = {2},
 pages = {1939--1963},
 year = {2025},
 language = {English},
 doi = {10.1137/23M1595382},
 zbMATH = {8043976},
 Zbl = {1565.35326}
}

@article {P23,
    AUTHOR = {Pankavich, Stephen},
     TITLE = {Scattering and asymptotic behavior of solutions to the
              {V}lasov--{P}oisson system in high dimension},
   JOURNAL = {SIAM J. Math. Anal.},
  FJOURNAL = {SIAM Journal on Mathematical Analysis},
    VOLUME = {55},
      YEAR = {2023},
    NUMBER = {5},
     PAGES = {4727--4750},
      ISSN = {0036-1410,1095-7154},
   MRCLASS = {35B40 (35A09 82D10)},
  MRNUMBER = {4645673},
MRREVIEWER = {Dingqun\ Deng},
       DOI = {10.1137/22M1520013},
       URL = {https://doi.org/10.1137/22M1520013},
}

@misc{HK24,
      title={Scattering of the {V}lasov--{R}iesz system in the three dimensions}, 
      author={Wenrui Huang and Hyunwoo Kwon},
      year={2024},
      eprint={2407.16919},
      archivePrefix={arXiv},
      primaryClass={math.AP},
      note={\url{https://arxiv.org/abs/2407.16919}}, 
}

@misc{Bed2022,
      title={A brief introduction to the mathematics of {L}andau damping}, 
      author={Jacob Bedrossian},
      year={2022},
      eprint={2211.13707},
      archivePrefix={arXiv},
      primaryClass={math.AP},
      note={\url{https://arxiv.org/abs/2211.13707}}, 
}

@article {CM98,
    AUTHOR = {Caglioti, E. and Maffei, C.},
     TITLE = {Time asymptotics for solutions of {V}lasov--{P}oisson equation
              in a circle},
   JOURNAL = {J. Statist. Phys.},
  FJOURNAL = {Journal of Statistical Physics},
    VOLUME = {92},
      YEAR = {1998},
    NUMBER = {1-2},
     PAGES = {301--323},
      ISSN = {0022-4715,1572-9613},
   MRCLASS = {82D10},
  MRNUMBER = {1645659},
       DOI = {10.1023/A:1023055905124},
       URL = {https://doi.org/10.1023/A:1023055905124},
}

@article{BCGIR24,
      title={Scattering problem for {V}lasov--type equations on the {$d$}--dimensional torus with {G}evrey data}, 
      author={Dario Benedetto and Emanuele Caglioti and Antoine Gagnebin and Mikaela Iacobelli and Stefano Rossi},
      year={2025},
      journal={Ann. Inst. H. Poincaré C Anal. Non Linéaire},
      note={Published online first},
      doi={10.4171/AIHPC/159},
}

@article {BCR22,
    AUTHOR = {Benedetto, Dario and Caglioti, Emanuele and Rossi, Stefano},
     TITLE = {Comparison between the {C}auchy problem and the scattering
              problem for the {L}andau damping in the {V}lasov--{HMF}
              equation},
   JOURNAL = {Asymptot. Anal.},
  FJOURNAL = {Asymptotic Analysis},
    VOLUME = {129},
      YEAR = {2022},
    NUMBER = {2},
     PAGES = {215--238},
      ISSN = {0921-7134,1875-8576},
   MRCLASS = {35Q83},
  MRNUMBER = {4465917},
       DOI = {10.3233/asy-211726},
       URL = {https://doi.org/10.3233/asy-211726},
}

@article {G24,
    AUTHOR = {Gagnebin, Antoine},
     TITLE = {Backward problem for the 1{D} ionic {V}lasov--{P}oisson
              equation},
   JOURNAL = {Kinet. Relat. Models},
  FJOURNAL = {Kinetic and Related Models},
    VOLUME = {17},
      YEAR = {2024},
    NUMBER = {2},
     PAGES = {312--330},
      ISSN = {1937-5093,1937-5077},
   MRCLASS = {82D10 (35B40 35Q83)},
  MRNUMBER = {4714592},
       DOI = {10.3934/krm.2023024},
       URL = {https://doi.org/10.3934/krm.2023024},
}

@book{S1970,
    author = {Stein, Elias} ,
    title = {Singular integrals and differentiability properties of functions},
    publisher = {Princeton University Press},
    year = {1970},
}

@book{T06,
    author = {T. Tao},
    title = {Nonlinear dispersive equations. Local and global analysis},
    publisher = {Providence, RI: American Mathematical Society (AMS)},
    year = {2006},
}

@Article{Penrose60,
 Author = {Penrose, Oliver},
 Title = {Electrostatic instabilities of a uniform non--{Maxwellian} plasma},
 FJournal = {Physics of Fluids},
 Journal = {Phys. Fluids},
 ISSN = {0031-9171},
 Volume = {3},
 Pages = {258--265},
 Year = {1960},
 Language = {English},
 DOI = {10.1063/1.1706024},
 zbMATH = {3147550},
 Zbl = {0090.22801}
}

@article {D22,
    AUTHOR = {Duan, Xianglong},
     TITLE = {Sharp decay estimates for the {V}lasov--{P}oisson and
              {V}lasov--{Y}ukawa systems with small data},
   JOURNAL = {Kinet. Relat. Models},
  FJOURNAL = {Kinetic and Related Models},
    VOLUME = {15},
      YEAR = {2022},
    NUMBER = {1},
     PAGES = {119--146},
      ISSN = {1937-5093,1937-5077},
   MRCLASS = {35Q83 (35M31)},
  MRNUMBER = {4389619},
MRREVIEWER = {Jinyeop\ Lee},
       DOI = {10.3934/krm.2021049},
       URL = {https://doi.org/10.3934/krm.2021049},
}

@article {S16,
    AUTHOR = {Smulevici, Jacques},
     TITLE = {Small data solutions of the {V}lasov--{P}oisson system and the
              vector field method},
   JOURNAL = {Ann. PDE},
  FJOURNAL = {Annals of PDE. Journal Dedicated to the Analysis of Problems
              from Physical Sciences},
    VOLUME = {2},
      YEAR = {2016},
    NUMBER = {2},
     PAGES = {Art. 11, 55},
      ISSN = {2524-5317,2199-2576},
   MRCLASS = {35Q83 (35M31)},
  MRNUMBER = {3595457},
MRREVIEWER = {Jonathan\ Ben-Artzi},
       DOI = {10.1007/s40818-016-0016-2},
       URL = {https://doi.org/10.1007/s40818-016-0016-2},
}

@misc{BVR24,
      title={Late--time asymptotics of small data solutions for the {V}lasov--{P}oisson system}, 
      author={Léo Bigorgne and Renato Velozo Ruiz},
      year={2024},
      eprint={2404.05812},
      archivePrefix={arXiv},
      primaryClass={math.AP},
      note={\url{https://arxiv.org/abs/2404.05812}}, 
}

@Article{CK16,
 Author = {Choi, Sun-Ho and Kwon, Soonsik},
 Title = {Modified scattering for the {Vlasov}--{Poisson} system},
 FJournal = {Nonlinearity},
 Journal = {Nonlinearity},
 ISSN = {0951-7715},
 Volume = {29},
 Number = {9},
 Pages = {2755--2774},
 Year = {2016},
 Language = {English},
 DOI = {10.1088/0951-7715/29/9/2755},
 zbMATH = {6629811},
 Zbl = {1351.82041}
}

@article {B91,
    AUTHOR = {Bouchut, Francois},
     TITLE = {Global weak solution of the {V}lasov--{P}oisson system for
              small electrons mass},
   JOURNAL = {Comm. Partial Differential Equations},
  FJOURNAL = {Communications in Partial Differential Equations},
    VOLUME = {16},
      YEAR = {1991},
    NUMBER = {8-9},
     PAGES = {1337--1365},
      ISSN = {0360-5302,1532-4133},
   MRCLASS = {82D10 (35Q99)},
  MRNUMBER = {1132788},
MRREVIEWER = {Reinhard\ Illner},
       DOI = {10.1080/03605309108820802},
       URL = {https://doi.org/10.1080/03605309108820802},
}

@article {HKI17,
    AUTHOR = {Han-Kwan, Daniel and Iacobelli, Mikaela},
     TITLE = {The quasineutral limit of the {V}lasov--{P}oisson equation in
              {W}asserstein metric},
   JOURNAL = {Commun. Math. Sci.},
  FJOURNAL = {Communications in Mathematical Sciences},
    VOLUME = {15},
      YEAR = {2017},
    NUMBER = {2},
     PAGES = {481--509},
      ISSN = {1539-6746,1945-0796},
   MRCLASS = {35Q83 (35B35)},
  MRNUMBER = {3620566},
MRREVIEWER = {Juhi\ Jang},
       DOI = {10.4310/CMS.2017.v15.n2.a8},
       URL = {https://doi.org/10.4310/CMS.2017.v15.n2.a8},
}

@article {HK11,
    AUTHOR = {Han-Kwan, Daniel},
     TITLE = {Quasineutral limit of the {V}lasov--{P}oisson system with
              massless electrons},
   JOURNAL = {Comm. Partial Differential Equations},
  FJOURNAL = {Communications in Partial Differential Equations},
    VOLUME = {36},
      YEAR = {2011},
    NUMBER = {8},
     PAGES = {1385--1425},
      ISSN = {0360-5302,1532-4133},
   MRCLASS = {35Q83 (76X05 82D10)},
  MRNUMBER = {2825596},
       DOI = {10.1080/03605302.2011.555804},
       URL = {https://doi.org/10.1080/03605302.2011.555804},
}

@article {CI23,
    AUTHOR = {Cesbron, Ludovic and Iacobelli, Mikaela},
     TITLE = {Global well-posedness of {V}lasov--{P}oisson--type systems in
              bounded domains},
   JOURNAL = {Anal. PDE},
  FJOURNAL = {Analysis \& PDE},
    VOLUME = {16},
      YEAR = {2023},
    NUMBER = {10},
     PAGES = {2465--2494},
      ISSN = {2157-5045,1948-206X},
   MRCLASS = {35Q82 (35Q83 76N10 82C70)},
  MRNUMBER = {4678146},
       DOI = {10.2140/apde.2023.16.2465},
       URL = {https://doi.org/10.2140/apde.2023.16.2465},
}

@Article{CS2011,
 Author = {Choi, Sun-Ho and Ha, Seung-Yeal and Lee, Ho},
 Title = {Dispersion estimates for the two-dimensional {Vlasov}-{Yukawa} system with small data},
 FJournal = {Journal of Differential Equations},
 Journal = {J. Differ. Equations},
 ISSN = {0022-0396},
 Volume = {250},
 Number = {1},
 Pages = {515--550},
 Year = {2011},
 Language = {English},
 DOI = {10.1016/j.jde.2010.10.005},
 zbMATH = {5834244},
 Zbl = {1252.35062}
}

@Article{SL2007,
 Author = {Ha, Seung-Yeal and Lee, Ho},
 Title = {Global well posedness of the relativistic {Vlasov}-{Yukawa} system with small data},
 FJournal = {Journal of Mathematical Physics},
 Journal = {J. Math. Phys.},
 ISSN = {0022-2488},
 Volume = {48},
 Number = {12},
 Pages = {123508, 19},
 Year = {2007},
 Language = {English},
 DOI = {10.1063/1.2820988},
 zbMATH = {5379084},
 Zbl = {1153.81369}
}

@article{Wei24,
      title={Nonlinear stability of the one dimensional screened {Vlasov} {Poisson} equation}, 
      author={Wei, Dongyi},
      year={2025},
      journal={Commun. Pure. Appl. Anal.},
      doi={10.3934/cpaa.2025054}
}
\end{document}